\documentclass{article}
\usepackage{arxiv}

\usepackage[utf8]{inputenc}
\usepackage[T1]{fontenc}

\usepackage{amsmath,amssymb,amsfonts}
\usepackage{mathtools}
\usepackage{amsthm}
\usepackage{physics}

\usepackage{graphicx}
\usepackage{subcaption}
\usepackage[font=small,labelfont=bf]{caption}

\usepackage{algorithm}
\usepackage{algpseudocode}
\algnewcommand{\IIf}[1]{\State\algorithmicif\ #1\ \algorithmicthen}
\algnewcommand{\EndIIf}{\unskip\ \algorithmicend\ \algorithmicif}

\usepackage[backend=biber,style=numeric, sorting=none]{biblatex}
\usepackage{xspace}
\usepackage[colorlinks=true,linkcolor=black,citecolor=black,urlcolor=magenta]{hyperref}
\usepackage{bm} 
\usepackage{afterpage}
\usepackage{booktabs}
\usepackage{cleveref}

\renewcommand{\vec}[1]{\bm{#1}}  
\newcommand{\mat}[1]{\bm{#1}}  
\DeclareMathOperator{\spn}{span}

\newcommand{\eye}{\mat{I}}

\newcommand{\urb}{\vec{\delta u}^N_{n+1}}
\renewcommand{\tr}{T}
\newcommand{\basis}{\mat{V}_n}
\newcommand{\tildebasis}{\tilde{\mat{V}}_n}
\newcommand{\subspace}{\mathcal{V}_n}

\newcommand{\inv}{^{-1}}

\newcommand{\imex}{IMEX--RB\xspace}

\newtheorem{theorem}{Theorem}[section]
\newtheorem{lemma}[theorem]{Lemma}
\newtheorem{proposition}[theorem]{Proposition}
\newtheorem{corollary}[theorem]{Corollary}

\newtheorem{remark}[theorem]{Remark}

\title{\imex: a self--adaptive implicit--explicit\\ time integration scheme exploiting\\ the reduced basis method}

\author{}

\date{}

\renewcommand{\headeright}{\imex: a self--adaptive IMEX time integration scheme exploiting the RB method}
\renewcommand{\undertitle}{}

\afterpage{\cfoot{\thepage}}

\begin{document}

\maketitle

\vspace{-2.5cm}

\begin{center}
\textbf{Micol Bassanini\footnote{\emph{Corresponding author.} Email: \href{mailto:micol.bassanini@epfl.ch}{\texttt{micol.bassanini@epfl.ch}}} \hspace{.5 cm} Simone Deparis \hspace{.5cm} Francesco Sala \hspace{.5 cm} Riccardo Tenderini}

\bigskip

\textit{Institute of Mathematics, EPFL, Lausanne, Switzerland}
\end{center}

\bigskip\bigskip

\begin{abstract}
In this work, we introduce a self--adaptive implicit--explicit (IMEX) time integration scheme, named \imex, for the numerical integration of systems of ordinary differential equations (ODEs), arising from spatial discretizations of partial differential equations (PDEs) by finite difference methods. Leveraging the Reduced Basis (RB) method, at each timestep we project the high--fidelity problem onto a suitable low--dimensional subspace and integrate its dynamics implicitly. Following the IMEX paradigm, the resulting solution then serves as an educated guess within a full--order explicit step. 
Notably, compared to the canonical RB method, \imex neither requires a parametrization of the underlying PDE nor features an offline--online splitting, since the reduced subspace is built dynamically, exploiting the high--fidelity solution history.
We present the first--order formulation of \imex, demonstrating and showcasing its convergence and stability properties. In particular, under appropriate conditions on the method's hyperparameters, \imex is unconditionally stable. 
The theoretical analysis is corroborated by numerical experiments performed on representative model problems in two and three dimensions. The results demonstrate that our approach can outperform conventional time integration schemes like backward Euler. Indeed, \imex yields high--fidelity accurate solutions, provided that its main hyperparameters --- namely the reduced basis size and the stability tolerance --- are suitably tuned. 
Moreover, \imex realizes computational gains over backward Euler for a range of timestep sizes above the forward Euler stability threshold.
\end{abstract}



\section{Introduction}
\label{sec: introduction}

Stability is a critical aspect when applying time--marching schemes to stiff dynamical systems, namely systems characterized by widely varying timescales and whose numerical integration is therefore challenging. In this work, we specifically focus on ordinary differential equations (ODEs) that stem from spatial discretizations of partial differential equations (PDEs) by finite difference schemes.

Explicit time integration methods, such as forward Euler, may represent an attractive choice for many problems, since they entail little computational cost and feature relatively straightforward implementations. 
However, for stiff systems, their use becomes impractical due to strict stability constraints that require choosing very small timestep sizes~\cite{lambert1991numerical, griffiths2010numerical,courant1928partiellen}. Nonetheless, these constraints are dictated by the fastest timescales, which may not be relevant to the physical phenomena of interest and may arise from transient and/or numerically induced effects \cite{durran1998numerical}. As a result, explicit methods often lead to overly accurate solutions, thereby involving excessively large computational costs.

Implicit methods, such as backward Euler, offer superior stability properties compared to explicit schemes, hence allowing the use of larger timesteps. Therefore, they are generally more suitable for the numerical approximation of stiff dynamical systems. Although they require solving a system of equations --- possibly nonlinear --- at each time step, which inevitably entails larger computational costs and implementation efforts. As a result, implicit schemes trade larger complexity for enhanced stability \cite{hundsdorfer2013numerical}. 

In this scenario, hybrid methods, called Implicit--Explicit (IMEX), have been designed~\cite{ascher1995implicit,ascher1997implicit}. In essence, IMEX methods aim at combining the advantages of explicit and implicit schemes, featuring remarkable stability properties and convenient costs.
A relevant example is represented by additive and partitioned Runge–Kutta methods \cite{hofer1976partially,kennedy2003additive}.
By integrating stiff terms implicitly and non--stiff terms explicitly, these schemes achieve superior efficiency compared to their explicit counterparts, relaxing timestep size constraints \cite{kennedy2003additive}. However, their performance heavily depends on how the system is partitioned, which requires \textit{a priori} knowledge of the underlying physics. In particular, their design is challenging for complex systems, characterised by a large number of equations and variables, such as two--fluid turbulence plasma model \cite{braginskii1965Transport} or gyrokinetic simulations \cite{doerk2014towards}. Indeed, in these scenarios, the interactions between different physical processes, and specifically the interplay of fast and slow dynamics, make it difficult to distinguish terms based on their stiffness.

An alternative splitting technique, often associated with exponential methods, is \textit{dynamic linearization}, also known as \textit{Jacobian splitting}~\cite{preuss2022when}. Notably, this approach does not require any prior knowledge of the problem's physics. By adding and subtracting the right--hand side operator Jacobian, dynamic linearization partitions into a stiff linear term, given by the right--hand side Jacobian evaluated at the most recent timestep, and a convenient Lipschitz--continuous nonlinear term, characterized by a small Lipschitz constant. Then, the latter is integrated explicitly, while the former, being stiff, is treated implicitly~\cite{luan2014exponential}. Albeit showcasing remarkable performances, we note that this approach nonetheless entails a high--fidelity implicit step and so non--negligible computational costs.

To circumvent the issues related to physics--based splitting, while at the same time retaining the stability properties of IMEX schemes, in this work we present a novel self--adaptive IMEX time integrator, named \imex. Remarkably, the splitting strategy underlying \imex does not require any prior knowledge of the problem physics nor of its spectral properties; furthermore, the splitting is self--adaptive as it is dynamically updated as time integration proceeds. Ultimately, the proposed scheme balances efficiency and robustness, offering high--fidelity accurate solution approximations at convenient computational costs, lower than those of conventional implicit methods for a wide range of timestep sizes. 

The workhorse of \imex is the Reduced Basis (RB) method~\cite{quarteroni2015reduced, hesthaven2016certified}, the most popular example of projection--based reduced order model (ROM). Compared to traditional high--fidelity models, the main idea of ROMs is to deliver precise enough solution approximations, whose quality actually depends on the task at hand, while drastically lowering the computational costs. 
Projection--based ROMs, in particular, achieve this goal by projecting the full--order problem, namely the one stemming from the spatial or spatio--temporal discretization of the PDE at hand, onto a suitably defined low--dimensional subspace, where the dynamics are evolved. This approach is particularly convenient when dealing with parametrized PDEs, i.e.\ PDEs that depend on one or more parameters, such as physical properties (e.g. density, viscosity) or geometrical features. Indeed, in this context, the solution manifold is inherently finite--dimensional and it can be approximated exploiting \textit{a priori} knowledge of a select number of snapshots and of the corresponding parameter values. Specifically, the manifold is approximated by a low--dimensional linear subspace, spanned by basis functions that are typically computed through Proper Orthogonal Decomposition \cite{kunisch2001galerkin, kunisch2002galerkin} or greedy approaches \cite{binev2011convergence, buffa2012priori}.
In recent years, significant progress has been made in designing RB methods for time--dependent nonlinear problems that are accurate, efficient, and stable \cite{nguyen2009reduced, yano2014space, choi2019space, hesthaven2022reduced, tenderini2024space}.
It is worth remarking that the effectiveness of RB depends on how precisely the PDE solution manifold can be fitted into a low--dimensional linear subspace, which can be expressed through the Kolmogorov \(n\)--width \cite{pinkus2012n}. In this regard, many problems, such as advection--dominated systems, wave--like equations and conservation laws, are well--known to feature large Kolmogorov \(n\)--widths and thus represent common pitfalls for the RB method.

Recently, major effort has also been devoted to the implementation of methods that leverage low--rank structures, commonly found in high--dimensional PDE solutions, to accelerate computations while preserving the accuracy, stability, and robustness of classical approaches. Popular examples in this regard are represented by the Dynamical Low Rank (DLR) approximation~\cite{koch2007dynamical} and by Step--and--Truncation (SAT) methods, \cite{rodgers2020step}.
For instance, within the DLR framework, several robust and accurate time integration techniques have been developed, including Projection DLR methods~\cite{kieri2019projection} and the Basis Updating and Galerkin (BUG) integrator~\cite{ceruti2022unconventional}.
By analogy with the \imex method, both the augmented BUG~\cite{ceruti2022rank} and the Reduced Augmentation Implicit Low--rank (RAIL)~\cite{nakao2025reduced} scheme dynamically update a low--dimensional basis over time.

In a nutshell, \imex operates in two stages. In the first stage, the RB method is employed to approximate the solution efficiently; notably, the small dimension of the reduced problem enables implicit time integration at a negligible computational cost. In the second stage, the prediction obtained from the reduced system is leveraged as an educated guess within a full--order explicit step, obeying the IMEX paradigm. The reduced subspace is generated through QR decomposition of the recent solution history matrix, and it is therefore dynamically updated over time. Furthermore, the reduced basis is suitably enriched at each timestep to guarantee absolute stability. 
Unlike conventional DLR time integration methods, \imex does not assume that the solution is inherently low--rank. Rather, our approach adaptively augments the reduced subspace so that it approximately contains the solution at the current timestep.

The paper is organized as follows. \Cref{sec:methods} presents the first--order \imex method, as resulting from the combination of the forward and backward Euler schemes, and demonstrates its convergence and absolute stability. Moreover, the complexity of the proposed algorithm is discussed and compared with that of backward Euler. In \cref{sec:numerical results}, the performance of the method is empirically evaluated on three different numerical experiments. Specifically, we use \imex to solve a 2D advection--diffusion equation, a 2D viscous Burgers' equation, and a 3D advection--diffusion equation. Finally, in \cref{sec: conclusion} we draw the conclusions, we list the main limitations and we discuss possible further developments.

\section{Main results}
\label{sec:methods}

In this section, we provide a detailed description of the proposed \imex method, as resulting from the convenient combination of a reduced backward Euler scheme and a high-fidelity explicit step. We initially describe the structure of \imex, highlighting its more relevant features. Subsequently, we demonstrate that, under suitable hypotheses, the proposed method is first--order convergent and absolutely stable. Finally, we report the complete algorithm with a few considerations on the associated computational costs and memory requirements. 

\subsection{Problem setup}
\label{subsec: problem setup}
We consider systems of ODEs that typically arise from the semi--discretization in space of PDEs by finite difference methods. In particular, we consider the following Cauchy problem:

find \(\vec{y} : I \subset \mathbb{R} \to \mathbb{R}^{N_h}\), \(\vec{y} \in C^2(I)\), such that
\begin{equation}\label{eq:cauchy_problem}
    \begin{cases}
        \vec{y}'(t) = \vec{f}(t,\vec{y}(t)), \quad \forall t \in I\\
        \vec{y}(0) = \vec{y}_0,
    \end{cases}
\end{equation}
where \(I = (0, T]\) is the time interval, \(\vec{f} : I \times \mathbb{R}^{N_h} \to \mathbb{R}^{N_h}\) is a known function and \(\vec{y}'\) is the first derivative of \(\vec{y}\) with respect to \(t\). Moreover, \(\vec{y}_0 \in \mathbb{R}^{N_h}\) is the initial condition. Here \(N_h \in \mathbb{N}\) denotes the number of degrees of freedom (DOFs), stemming from the discretization in space of the underlying PDE. The problem in \cref{eq:cauchy_problem} is well--posed if \(\vec{f}\) is bounded, continuous in both arguments, and Lipschitz continuous in its second argument.
These assumptions are understood in the following derivation; we further denote by \(L >0\) the Lipschitz constant of \(\vec{f}\).

\subsection{The \imex method}
\label{subsec: the imex-rb method}
For a given \(N_t \in \mathbb{N}\), we build a sequence \(\{\vec{u}_n\}_{n=0}^{N_t}\) approximating \(\vec{y}(t_n)\), \(n=0, \ldots, N_t\), where \(t_n = n\Delta t\) and \(\Delta t \coloneqq T/N_t\) is the timestep, i.e. we assume \(t_0 = 0\) for simplicity.

The proposed \imex scheme allows for a stable and efficient numerical integration of ODEs systems by exploiting two steps. In the first one, a backward Euler step is projected onto a conveniently defined reduced subspace \(\subspace\) of dimension $N \ll N_h$, spanned by the orthonormal columns of a matrix \(\basis \in \mathbb{R}^{N_h \times N}\), whose assembly is the subject of \cref{subsec: the imex-rb algorithm}. In the second step, instead, we explicitly integrate \cref{eq:cauchy_problem} through a modified forward Euler step, employing the available intermediate RB solution as a convenient evaluation point for \(\vec{f}\). Formally, let \(\vec{u}_n \in \mathbb{R}^{N_h}\) be the approximate solution to \cref{eq:cauchy_problem} at time \(t_n\) computed with \imex. Also, let the reduced subspace be such that \(\vec{u}_n \in \subspace\). Let \(\vec{u}_0 = \vec{y}(0)\); then, for \(n = 0,1,\ldots, N_t-1\), we compute \(\vec{u}_{n+1}\) as follows:
\begin{enumerate}
\item Solve a projected backward Euler step in the reduced space, looking for \(\urb \in \mathbb{R}^N\), 
\begin{equation}
\label{eq: implicit step}
\urb = \Delta t \basis^{\tr} \vec{f}(t_{n+1}, \basis \urb + \vec{u}_n);
\end{equation}
\item Compute \(\vec{u}_{n+1}\) explicitly, by conveniently exploiting the solution to \cref{eq: implicit step}: 
\begin{equation}
\label{eq: explicit step}
\vec{u}_{n+1} = \vec{u}_n + \Delta t \vec{f}(t_{n+1}, \basis \urb + \vec{u}_{n}).
\end{equation}
\end{enumerate}
In the following derivations, we assume that \(\vec{u}_n \in \subspace\). Additional details about the definition and construction of the subspace $\subspace$ are discussed in \cref{subsec: the imex-rb algorithm}.

\begin{lemma}
    The scheme represented by \cref{eq: implicit step,eq: explicit step} can be equivalently written in the form:
\begin{equation}\label{eq:imex_scheme_rewritten}
     \vec{u}_{n+1} = \vec{u}_n + \Delta t \vec{f}(t_{n+1}, \basis \basis^{\tr} \vec{u}_{n+1}).
\end{equation}
\end{lemma}
\begin{proof}
    We first show that from \cref{eq: implicit step,eq: explicit step}, \cref{eq:imex_scheme_rewritten} can be derived.  On the one hand, since  \(\vec{u}_n \in \subspace\), it follows that \(\vec{u}_n = \basis \basis^{\tr} \vec{u}_n\), and thus projecting \cref{eq: explicit step} onto the subspace \(\subspace\) yields: 
\begin{equation*}
   \basis \basis^{\tr} \vec{u}_{n+1} = \vec{u}_n + \Delta t \basis \basis^{\tr} \vec{f}(t_{n+1}, \basis \urb + \vec{u}_n).
\end{equation*}
On the other hand, the reduced coefficients \(\urb \in \mathbb{R}^N\) satisfy the backward Euler step of \cref{eq: implicit step}. Pre--multiplying by \(\basis\) and adding \(\vec{u}_n\) on both sides yields:
\begin{equation*}
\basis \urb + \vec{u}_n = \vec{u}_n + \Delta t \basis \basis^{\tr} \vec{f}(t_{n+1},\basis \urb+\vec{u}_n).
\end{equation*}
This confirms that \(\basis \basis^{\tr} \vec{u}_{n+1} = \basis \urb + \vec{u}_n\), thus showing that the right--hand sides of \cref{eq: explicit step,eq:imex_scheme_rewritten} are equal. This proves \cref{eq:imex_scheme_rewritten}.

Now, let \cref{eq:imex_scheme_rewritten} hold. Take
$$ \vec{\delta u}_{n+1} \coloneqq \vec{u}_{n+1} - \vec{u}_n = \Delta t \ \vec{f}(t_{n+1}, \basis \basis^{\tr} \vec{u}_{n+1})~.$$
Then
$$ \basis^{\tr} \vec{\delta u}_{n+1} = \Delta t \basis^{\tr} \vec{f}(t_{n+1}, \basis \basis^{\tr} \vec{u}_{n+1}),$$
and
$$ \basis \basis^{\tr} \vec{u}_{n+1} = \basis \basis^{\tr} (\vec u_n + \vec{\delta u}_{n+1}) = \vec u_n + \basis^{\tr} \vec{\delta u}_{n+1}.$$
Upon defining $\vec{\delta u}_{n+1}^N \coloneqq \basis^{\tr} \vec{\delta u}_{n+1}$, we obtain \cref{eq: implicit step,eq: explicit step}.
\end{proof}

\subsection{Convergence of \imex}
\label{subsec: convergence of imex-rb}
First, we investigate the consistency of \imex.
\begin{proposition}[Consistency of \imex]\label{prop:consistency}
    Consider the Cauchy problem of \cref{eq:cauchy_problem}. Let \(\vec{y} \in C^2(I)\), and suppose that \(\vec{f}\) is Lipschitz continuous in its second argument, with Lipschitz constant \(L>0\). Then, the \imex\ scheme is consistent of order 1.
\end{proposition}

\begin{proof}
We proceed as in \cite[p.~211]{quarteroni2010scientific}. Consider the equivalent reformulation of the \imex scheme in \cref{eq:imex_scheme_rewritten}.
Let \(\vec{\tau}_{n+1}(\Delta t)\) be the local truncation error at time \(t_{n+1}\).
From its definition, and by imposing that the exact solution \(\vec{y}\) satisfies the scheme, we can write
\begin{equation}\label{eq:consistency_step1}
    \vec{y}(t_{n+1}) = \vec{y}(t_n) + \Delta t\, \vec{f}\bigl(t_{n+1}, \tildebasis \tildebasis^{\tr} \vec{y}(t_{n+1})\bigr) + \Delta t\, \vec{\tau}_{n+1}(\Delta t),
\end{equation}
where the notation \(\tildebasis\) accounts for the presence of the exact solution \(\vec{y}(t_n)\) inside the basis.
A rearranged first-order Taylor expansion of \(\vec{y}(t_n)\) around \(t_{n+1}\) yields:
\begin{equation}\label{eq:consistency_step2}
    \vec{y}(t_{n+1}) = \vec{y}(t_n) + \Delta t\, \vec{f}\bigl(t_{n+1}, \vec{y}(t_{n+1})\bigr) - \frac{1}{2} \Delta t^2\, \vec{y}''(\xi), \quad \xi \in [t_n, t_{n+1}].
\end{equation}
Subtracting \cref{eq:consistency_step1} from \cref{eq:consistency_step2} gives:
\begin{equation*}
  \vec{\tau}_{n+1}(\Delta t)
    = \vec{f}\bigl(t_{n+1}, \vec{y}(t_{n+1})\bigr) -  \vec{f}\bigl(t_{n+1}, \tildebasis \tildebasis^{\tr} \vec{y}(t_{n+1})\bigr) - \frac{1}{2} \Delta t\, \vec{y}''(\xi).
\end{equation*}
Taking the Euclidean norms and leveraging the Lipschitz continuity of \(\vec{f}\) in its second argument, we obtain:
\begin{equation*}
    \norm{\vec{\tau}_{n+1}(\Delta t)} \leq L \norm{\left(\eye - \tildebasis \tildebasis^{\tr}\right) \vec{y}(t_{n+1})} + \frac{1}{2} \Delta t\, \norm{\vec{y}''(\xi)}.
\end{equation*}
Since by hypothesis \(\vec{y}(t_n)\) lies in the range of \(\tildebasis\), and thanks to \cref{eq:consistency_step2}, it holds that
\begin{align*}
    \left(\eye - \tildebasis \tildebasis^{\tr}\right) \vec{y}(t_{n+1})
    &= \left(\eye - \tildebasis \tildebasis^{\tr}\right) \bigl( \vec{y}(t_{n+1}) - \vec{y}(t_n) \bigr) \\
    &= \Delta t \left(\eye - \tildebasis \tildebasis^{\tr}\right) \vec{f}\bigl(t_{n+1}, \vec{y}(t_{n+1})\bigr)+\\ 
    &\quad - \frac{1}{2} \Delta t^2 \left(\eye - \tildebasis \tildebasis^{\tr}\right) \vec{y}''(\xi).
\end{align*}
Thus, the following estimate holds:
\begin{equation*}
    \norm{\vec{\tau}_{n+1}(\Delta t)}
    \leq \Delta t\, L \norm{\vec{f}\bigl(t_{n+1}, \vec{y}(t_{n+1})\bigr)}
         + \frac{1}{2} \Delta t^2\, L \norm{\vec{y}''(\xi)}
         + \frac{1}{2} \Delta t \norm{\vec{y}''(\xi)}.
\end{equation*}
Therefore, the global truncation error \(\tau(\Delta t) \coloneqq \max_n \norm{\vec{\tau}_{n+1}(\Delta t)}\) satisfies:
\begin{equation*}
    \tau(\Delta t) \leq \Delta t\, L \max_{t \in I} \norm{\vec{y}'(t)} + \frac{1}{2} \bigl( \Delta t +  \Delta t^2 L \bigr) \max_{t \in I} \norm{\vec{y}''(t)}.
\end{equation*}
Hence, since \(\vec{y} \in C^2(I)\), we conclude that \(\tau(\Delta t) = \mathcal{O}(\Delta t)\), and the \imex scheme is first-order consistent \cite[p.~484]{quarteroni2006numerical}.
\end{proof}

\begin{theorem}[Convergence of \imex]\label{theo:convergence}
Under the same assumptions of \cref{prop:consistency}, the \imex method is convergent of order 1, i.e. \(\exists C > 0\) such that
\begin{equation*}
    \norm{\vec{e}_{n+1}} \coloneqq \norm{\vec{y}(t_{n+1}) - \vec{u}_{n+1}} < C \Delta t.
\end{equation*}
\end{theorem}
\begin{proof}
We decompose the global error at \(t_{n+1}\) into two parts \cite[p.~487]{quarteroni2006numerical}:
\begin{equation*}
     \vec{e}_{n+1} = (\vec{y}(t_{n+1}) - \vec{u}^*_{n+1}) + (\vec{u}^*_{n+1}-\vec{u}_{n+1}),
\end{equation*}
where 
\begin{equation*}
      \vec{u}^{*}_{n+1} \coloneqq \vec{y}(t_n) + \Delta t\, \vec{f}\bigl(t_{n+1}, \tildebasis \tildebasis^{\tr} \vec{y}(t_{n+1})\bigr).
\end{equation*}
Taking norms and using consistency gives
\begin{equation}\label{eq:error_components}
        \norm{\vec{e}_{n+1}} \leq \norm{\vec{y}(t_{n+1}) - \vec{u}^*_{n+1}} + \norm{\vec{u}^*_{n+1}-\vec{u}_{n+1}} \leq \Delta t \tau(\Delta t) + \norm{\vec{u}^*_{n+1}-\vec{u}_{n+1}}.
\end{equation}
We now bound the quantity \(\norm{\vec{u}^*_{n+1} - \vec{u}_{n+1}}\).  First, we observe that
\begin{equation*}
    \vec{u}^*_{n+1} - \vec{u}_{n+1} =  \vec{y}(t_n) -  \vec{u}_n + \Delta t [\vec{f}(t_{n+1}, \tildebasis\tildebasis^{\tr} \vec{y}(t_{n+1}) ) - \vec{f}(t_{n+1}, \basis \basis^{\tr} \vec{u}_{n+1}) ].
\end{equation*}
By the Lipschitz continuity of \(\vec f\), it follows that
\begin{align*}
  \norm{\vec{u}^*_{n+1}-\vec{u}_{n+1}} & \leq \norm{\vec{e}_n} +\Delta t \norm{\vec{f}(t_{n+1}, \tildebasis\tildebasis^{\tr} \vec{y}(t_{n+1}) ) - \vec{f}(t_{n+1}, \basis \basis^{\tr} \vec{u}_{n+1}) }\\
    &\leq \norm{\vec{e}_n} + \Delta t L \norm{\tildebasis\tildebasis^{\tr} \vec{y}(t_{n+1}) - \basis \basis^{\tr} \vec{u}_{n+1}}.
\end{align*}
We then write
\begin{equation*}
  \tildebasis\tildebasis^{\tr} \vec{y}(t_{n+1}) - \basis \basis^{\tr} \vec{u}_{n+1}
= (\tildebasis \tildebasis^{\tr} - \basis \basis^{\tr}) \vec{y}(t_{n+1}) + \basis \basis^{\tr} (\vec{y}(t_{n+1}) - \vec{u}_{n+1}).
\end{equation*}
By the triangle inequality and since
\begin{equation*}
    \norm{\tildebasis \tildebasis^{\tr} - \basis \basis^{\tr}} = \norm{(\tildebasis - \basis) \tildebasis^{\tr} +  \basis (\tildebasis - \basis)^{\tr}} \leq 2 \norm{\tildebasis - \basis},
\end{equation*} 
we get
\begin{equation}\label{eq:passage1}
\norm{\tildebasis\tildebasis^{\tr} \vec{y}(t_{n+1}) - \basis \basis^{\tr} \vec{u}_{n+1}}
\le
2\,\norm{\tildebasis - \basis}\,\norm{\vec y(t_{n+1})}
+\norm{\vec e_{n+1}}.
\end{equation}
We assume \(\tildebasis\) is obtained by normalizing \(\vec y(t_n)\) and \(\basis\) by normalizing \(\vec u_n\); hence,
\begin{equation}\label{eq:assumption_basis_convergence}
    \tilde{\mat{V}}_n \coloneqq \frac{\vec{y}(t_n)}{\norm{\vec{y}(t_n)}},
\qquad
\basis \coloneqq \frac{\vec{u}_n}{\norm{\vec{u}_n}}.
\end{equation}
The fallback when \(\|\vec y(t_n)\|\) is below a prescribed tolerance appears in \cref{rmk:norm_y_0}. 
Now, exploiting that for two generic vectors $\vec a, \vec b \in \mathbb{R}^K, \ K \in \mathbb{N}$, it holds
\begin{equation*}
\norm{\frac{\vec b}{\|\vec b\|}-\frac{\vec a}{\|\vec a\|}}
\le
2\,\frac{\|\vec b-\vec a\|}{\|\vec b\|},
\end{equation*}
we obtain
\[
\norm{\tildebasis - \basis}
\le 2 \frac{\norm{\vec{e}_n}}{\norm{\vec{y}(t_n)}}.
\]
Substituting this result into \cref{eq:passage1} yields
\begin{equation}
\label{eq:to_be_replaced}
\norm{\tildebasis\tildebasis^{\tr} \vec{y}(t_{n+1}) - \basis \basis^{\tr} \vec{u}_{n+1}}
\le  4 \frac{\norm{\vec{y}(t_{n+1})}}{\norm{\vec{y}(t_n)}} \norm{\vec{e}_n} + \norm{\vec{e}_{n+1}}.
\end{equation}
Hence, we have that
\begin{equation}
\label{eq: stability bound}
     \norm{\vec{u}^*_{n+1}-\vec{u}_{n+1}}  \leq (1+4\Delta t C L )\norm{\vec{e}_n} + \Delta t L  \norm{\vec{e}_{n+1}},
\end{equation}
where \(C = C(\Delta t) \coloneqq  \frac{\norm{\vec{y}(t_{n+1})}}{\norm{\vec{y}(t_n)}} \), and \(C(\Delta t) \to 1\) as \(\Delta t \to 0\).

Leveraging \cref{eq: stability bound} and assuming \(\Delta t\,L<1\), \cref{eq:error_components} becomes
\begin{equation*}
    \norm{\vec{e}_{n+1}} \leq \frac{1+4\Delta t C L }{1-\Delta t L}\norm{\vec{e}_n} + \frac{\Delta t}{1-\Delta t L} \tau(\Delta t).
\end{equation*}
Let us define \(\alpha \coloneqq 1- \Delta t L\) and \(\beta \coloneqq 1 + 4 \Delta t C L\). By recursion, we obtain:
\begin{equation*}
    \norm{\vec{e}_{n+1}} \leq  \alpha^{-1} \Delta t \tau(\Delta t) \sum_{j=0}^{n} \beta^{j} \alpha^{-j} +  (\beta \alpha^{-1})^{n+1}\norm{\vec{e}_0},
\end{equation*}
where \(\norm{\vec{e}_0} \equiv 0\), so the last term cancels out. To treat the sum, we first bound \(\beta^j \leq \beta^{N_t}\), \(\forall j\), and use the geometric sum for the remaining term in \(\alpha^{-j}\):
\begin{equation*}
    \norm{\vec{e}_{n+1}} \leq  \alpha^{-1} \beta^{N_t} \Delta t \tau(\Delta t) \frac{\alpha^{-(n+1)} -1}{\alpha^{-1}-1} = (\beta^{N_t} \Delta t \tau(\Delta t))  \alpha^{-1} \frac{\alpha^{-(n+1)} -1}{\alpha^{-1}-1}.
\end{equation*}
The terms in \(\alpha\) develop as follows:
\begin{equation*}
     \alpha^{-1} \frac{\alpha^{-(n+1)} -1}{\alpha^{-1}-1} =   \alpha^{-1} \frac{\alpha^{-(n+1)} -1}{\frac{\Delta t L}{\alpha}} = \frac{\alpha^{-(n+1)}-1}{\Delta t L} \leq \frac{e^{\Delta t (n+1) L}-1}{\Delta t L}
\end{equation*}
where we exploited the fact that \((1-x) < e^{-x}\). On the whole,
\begin{equation*}
     \norm{\vec{e}_{n+1}} \leq \beta^{N_t} \frac{e^{T L}-1}{L} \tau(\Delta t).
\end{equation*}
Finally, since \(\beta^{N_t} = (1+4C \Delta t L)^{N_t} \leq e^{4 C T L}\),
we have:
\begin{equation*}
      \norm{\vec{e}_{n+1}} \leq \frac{e^{T L}-1}{L} e^{4 C T L} \tau(\Delta t).
\end{equation*}
Therefore, the error is bounded by the product of the constant \(\frac{e^{T L}-1}{L}\) times an \(\mathcal{O}(\Delta t)\) term (\(e^{4 C T L} \tau(\Delta t)\)), which completes the proof.
\end{proof}
\begin{remark}\label{rmk:norm_y_0}
   In the proof, we assumed that \(\|\vec{y}(t_n)\|\) does not vanish.  When \(\|\vec y(t_n)\|\) falls below \(K_1\Delta t\) ($K_1 \in \mathbb{R}^+)$, we set \(\tildebasis=[1,0,\dots,0]^{\tr}\).  Then \(\|\tildebasis-\basis\|\le2\), and by continuity of the analytical solution, we have \(\|\vec y(t_{n+1})\|\le K_2\Delta t\) as well.  Therefore, \cref{eq:to_be_replaced} rewrites as
    \begin{equation*}
    \norm{\tildebasis\tildebasis^{\tr} \vec{y}(t_{n+1}) - \basis \basis^{\tr} \vec{u}_{n+1}}
\le 4 \norm{\vec{y}(t_{n+1})} + \norm{\vec{e}_{n+1}} \le 4 K_2 \Delta t + \norm{\vec{e}_{n+1}},
    \end{equation*}
   which yields
    \begin{equation*}
    \norm{\vec{e}_{n+1}} \leq \norm{\vec{e}_n} + \Delta t L  \norm{\vec{e}_{n+1}} + 4 K_2 L \Delta t^2 + \Delta t \tau(\Delta t).
\end{equation*}
        Now, the \(\mathcal{O}(\Delta t^2)\) term can be combined with the global truncation error to complete the proof as above.
\end{remark}

\bigskip

\begin{corollary}\label{prop:zero_stab}
    The \imex scheme is zero--stable.
\end{corollary}
\begin{proof}
    By the Lax--Richtmyer equivalence theorem \cite[p.~41]{quarteroni2006numerical}, a consistent and convergent method is zero--stable.
\end{proof}

\subsection{Absolute stability of \imex}
\label{subsec: absolute stability of imex-rb}
We study the absolute stability of \imex when applied to the model problem:
\begin{equation*}
    \begin{cases}
        \vec{y}'(t) = \mat{A} \vec{y}(t) \quad t \in I, \\
        \vec{y}(0) = \vec{y}_0,
    \end{cases}
\end{equation*}
where \(\mat{A} =  \mat{T} \mat{\Lambda} \mat{T}^{-1} \in \mathbb{R}^{N_h \times N_h}\) is supposed to be a diagonalizable matrix, where \(\mat{\Lambda} \in \mathbb{R}^{N_h \times N_h}\) is the diagonal matrix featuring the eigenvalues, and \(\mat{T} \in \mathbb{R}^{N_h \times N_h}\) is the matrix of eigenvectors. Besides, let the eigenvalues \(\{\lambda_i\}_{i=1}^{N_h} \in \mathbb{C}\) satisfy \(\Re\{\lambda_i\} < 0, \,\, i=1, \ldots, N_h\). By definition of absolute stability, we aim to prove:
\begin{equation*}
    \lim_{n \to + \infty}\norm{\vec{u}_{n+1}} = 0.
\end{equation*}

\begin{theorem}[Absolute stability of \imex]\label{theo:abs_stability}
Let \(\mat{A} = \mat{T} \mat{\Lambda} \mat{T}^{-1} \in \mathbb{R}^{N_h \times N_h}\) be a diagonalizable matrix, with \(N_h\) eigenvalues having negative real parts. Let \(\mu_{\mathrm{min}} \coloneqq \min_i |\Re\{\lambda_i\}|\) and \(\mu_{\mathrm{max}} \coloneqq \max_i |\lambda_i|\). Suppose that at each timestep \(t_n\) the reduced basis matrix \(\basis \in \mathbb{R}^{N_h \times N}\) satisfies:
\begin{equation}\label{eq:absolute_stab_cond}
    \norm{\left(\eye - \basis \basis^{\tr}\right) \vec{u}_{n+1}} \leq \varepsilon \norm{\vec{u}_{n+1}},
\end{equation}
where \(\varepsilon\) is such that
\begin{equation}\label{eq:bound_eps}
    0 \leq \varepsilon < \frac{\mu_{\mathrm{min}}}{\mu_{\mathrm{max}} K_2(\mat{T})} \leq 1.
\end{equation}
Here, \(K_2(\mat{T}) = \norm{\mat{T}} \|\mat{T}^{\inv}\|\) is the condition number of the eigenvectors matrix \(\mat{T}\) in the spectral norm.
Then the \imex scheme is absolutely stable.
\end{theorem}

\begin{proof}
We switch to modal coordinates by defining \(\vec{v}_{n+1} \coloneqq \mat{T}^{\inv} \vec{u}_{n+1}\). Left-multiplying the \imex scheme of \cref{eq:imex_scheme_rewritten} by \(\mat{T}^{\inv}\) yields:
\begin{equation*}
    \vec{v}_{n+1} = \vec{v}_n + \Delta t \mat{\Lambda} \mat{T}^{-1} \basis \basis^{\tr} \vec{u}_{n+1}.
\end{equation*}
This expression can be rewritten as:
\begin{equation*}
    \vec{v}_{n+1} = \vec{v}_n + \Delta t \mat{\Lambda} \vec{v}_{n+1} - \Delta t \mat{\Lambda} \mat{T}^{-1} \left(\eye - \basis \basis^{\tr}\right)\vec{u}_{n+1},
\end{equation*}
and thus
\begin{equation}\label{eq: to_be_bounded}
   \vec{v}_{n+1} = \left(\eye - \Delta t \mat{\Lambda}\right)^{\inv} [\vec{v}_n - \Delta t \mat{\Lambda}\mat{T}^{-1} \left(\eye - \basis \basis^{\tr}\right) \vec{u}_{n+1}].
\end{equation}
Let \(\mu_{\mathrm{min}}= \min_i |\Re\{\lambda_i\}| > 0\). Since for any \(z \in \mathbb{C}\), it holds \(\left|\frac{1}{z} \right| = \frac{1}{|z|} \leq \frac{1}{|\Re\{z\}|}\), we have
\begin{equation*}
    \norm{\left(\eye - \Delta t\,\mat{\Lambda}\right)^{-1}} = \max_i \left|\frac{1}{1 - \Delta t\,\lambda_i}\right| 
    \leq \frac{1}{1 + \Delta t\,\mu_{\mathrm{min}}}.
\end{equation*}
Taking the norm of \cref{eq: to_be_bounded}, exploiting the assumption in \cref{eq:absolute_stab_cond} and that \(\norm{\vec{u}_{n+1}} \leq \norm{\mat{T}} \norm{\vec{v}_{n+1}}\), we obtain:
\begin{equation*}
    \norm{\vec{v}_{n+1}} \leq \frac{1}{1+\Delta t \mu_{\mathrm{min}}} \Bigl[\norm{\vec{v}_n} + \Delta t \varepsilon \ \mu_{\mathrm{max}} \ K_2(\mat{T}) \ \norm{\vec{v}_{n+1}} \Bigr],
\end{equation*}
where we used \(\norm{\mat{\Lambda}} = \mu_{\mathrm{max}}\). Rearranging terms gives:
\begin{equation*}
    \left(1 - \frac{\Delta t \varepsilon \mu_{\mathrm{max}} K_2(\mat{T})}{1+\Delta t \mu_{\mathrm{min}}} \right) \norm{\vec{v}_{n+1}} \leq \frac{1}{1 + \Delta t \mu_{\mathrm{min}}} \norm{\vec{v}_n}.
\end{equation*}
Under the condition \(\varepsilon < \frac{1+\Delta t \mu_{\mathrm{min}}}{\Delta t \mu_{\mathrm{max}} K_2(\mat{T})} \eqqcolon C_1\), the multiplicative factor in the left-hand side is positive and the bound is nontrivial. Therefore, the scheme is absolutely stable provided that:
\begin{equation}\label{eq:denom_abs_stability}
    \frac{1}{1+\Delta t \mu_{\mathrm{min}}- \Delta t \varepsilon \mu_{\mathrm{max}} K_2(\mat{T})} < 1 \quad \iff \quad \varepsilon < \frac{\mu_{\mathrm{min}}}{\mu_{\mathrm{max}} K_2(\mat{T})} \eqqcolon C_2.
\end{equation}
Since \(C_2 < C_1\), \(C_2\) constitutes the most restrictive bounding term.
Under the condition $\varepsilon < C_2$, we conclude that \(\norm{\vec{v}_{n+1}} \to 0\) as \(n \to +\infty\). Since \(\norm{\vec{u}_{n+1}} \leq \norm{\mat{T}} \norm{\vec{v}_{n+1}}\), it follows \(\lim_{n \to +\infty} \norm{\vec{u}_{n+1}} = 0\), which proves the absolute stability of the \imex scheme.
\end{proof}

The quantity $\varepsilon$ introduced in \cref{eq:absolute_stab_cond} is a stability tolerance. By analogy with the canonical RB method, it corresponds to the proper orthogonal decomposition error tolerance on $\vec u_{n+1}$.
If \(\mat{A}\) is symmetric, then \(\mu_{\mathrm{min}}= |\lambda_1|\), \(K_2(\mat{T}) = 1\); hence, \cref{eq:bound_eps} simplifies to \(\varepsilon \leq K_2(\mat{A})^{\inv} \eqqcolon \bar{\varepsilon}\).  As a result, the worse the condition number of \(\mat{A}\), the smaller the value of \(\varepsilon\) required to satisfy the inequality in \cref{eq:absolute_stab_cond}. If \(\mat{A}\) is not symmetric, the theoretical bound \(C_2\) reported in \cref{eq:denom_abs_stability} is always lower than \(\bar{\varepsilon}\). Since computing \(C_2\) is generally impractical, in our numerical tests we set $\varepsilon = \gamma \bar{\varepsilon}$, $\gamma \in \mathbb{R}^+$, to derive a computationally affordable educated guess.

\section{The \imex algorithm}
\label{subsec: the imex-rb algorithm}

In \cref{alg:imexrb}, we describe the \imex algorithm and its computational costs in detail. 

At time \(t_{n}\), we perform the steps of \cref{eq: implicit step,eq: explicit step}, to compute \(\vec{u}_{n+1}\). This requires assembling a basis \(\basis\) satisfying the inequality of \cref{eq:absolute_stab_cond} for a given value of \(\varepsilon\) subject to the condition of \cref{eq:bound_eps}. Aiming to meet the stability condition, and motivated by the intuition that past solutions can be exploited to compute an accurate extrapolation of the new solution \(\vec{u}_{n+1}\), we define the \(\basis\) as an orthonormal basis for the subspace \(\subspace = \spn\{\vec{u}_{n-N_n+1}, \dots, \vec{u}_n\}\), for some \(N_n \coloneqq \min\{N, n\} \in \mathbb{N}^+\) (\cref{alg_line:qr}). The user--defined value of \(N\) is problem--dependent: it should be increased for larger values of \(N_h\) or larger values of \(\Delta t\). As a rule of thumb, in the numerical tests we always select \(N\) to be on the order of tens, even for problems with up to \(10^5\) degrees of freedom.

However, choosing the so--defined reduced basis size may not be sufficient to satisfy the inequality of \cref{eq:absolute_stab_cond} for a given \(\varepsilon \in \mathbb{R}^+\). Therefore, we introduce a loop (\cref{alg_line:stop}), thus allowing for ``inner'' iterations, solving \cref{eq: implicit step,eq: explicit step} for \(k = 0,1, \dots\) (\cref{alg_line:im,alg_line:ex}) and generating a sequence of iterates \(\vec{u}_{n+1}^{(k)}\). Such iterates are used to conveniently enrich the reduced subspace, until \cref{eq:absolute_stab_cond} holds. To do so, we define an initial orthonormal basis \(\basis^{(0)}\) so that \(\subspace^{(0)} = \spn\{\vec{u}_{n-N_n+1}, \dots, \vec{u}_n\}\) and assemble the basis by means of a QR decomposition. Note that, especially for small values of \(\Delta t\), some \(\vec{u}_{n-i}\), \(i=1, \ldots, N_n-1\), might be collinear to \(\vec{u}_n\) or among themselves. Hence, starting from \(\vec{u}_n\), we include the vector \(\vec{u}_{n-i}\) in the reduced basis \(\basis^{(0)}\) only if the basis constructed up to \(\vec{u}_{n-i+1}\) and augmented with \(\vec{u}_{n-i}/ \norm{\vec{u}_{n-i}}\) has a reciprocal condition number greater than a given tolerance \(\delta \in \mathbb{R}^+\), as reported in \cref{alg_line:qr}.  

After each inner iteration, if the absolute stability criterion is not met (\cref{alg_line:n+1}), we augment the reduced basis with the orthogonal complement of the current iterate \(\vec{u}_{n+1}^{(k)}\), performing one step of the Gram--Schmidt orthogonalization procedure (\cref{alg_line:pi}). Hence, at the $k$--th inner iteration, the reduced basis matrix \(\basis^{(k)}\) has $N_k \coloneqq N_n + k$ columns. A maximum number \(M \in \mathbb{N}\) of inner iterations is allowed, so that the maximal size the reduced basis can attain is $\tilde{N}\coloneqq N_n+M-1$.
We highlight that \(M\) is a user--defined parameter, which should be set sufficiently large (e.g., \(M = 100\)). If the algorithm fails to satisfy the stability condition in \cref{eq:absolute_stab_cond} within \(M\) inner iterations, the procedure should be terminated, and larger values of \(M\) and/or \(N\) should be considered.

Once the absolute stability condition is satisfied, say after $K \in \mathbb{N}$, \(K \leq M\), inner iterations, we identify the approximate solution $\vec{u}_{n+1}$ as the last available iterate $\vec{u}_{n+1}^{(K)}$ (\cref{alg_line:new_sol}). Therefore, with reference to the previous theoretical analysis, we identify the subspace \(\subspace \coloneqq \subspace^{(K)}\), spanned by the columns of \(\basis \coloneqq \basis^{(K)}\). 

In practice, the QR decomposition in \cref{alg_line:qr} is performed through a QR update of the reduced basis matrix \(\mat{V}_{n-1}^{(0)} \in \mathbb{R}^{N_h \times N_{n-1}}\) inherited from the previous timestep, and it can be efficiently computed using well--established algorithms~\cite{daniel1976reorthogonalization,golub2013matrix}. Indeed, the QR factorization within \imex 
incurs a computational cost of just \(\mathcal{O}(N_h N_n)\) and is performed once per timestep. Furthermore, when adding \(\vec{u}_n\) to the basis for the vectors \([\vec{u}_{n-N_{n}-1}, \ldots, \vec{u}_{n-1}]\), we handle potential collinearities by setting a lower bound \(\delta = 10^{-8}\), which represents the value of the \texttt{rcond} parameter in the \texttt{QR\_insert} function of \emph{SciPy}\footnote{\url{https://docs.scipy.org/doc/scipy/reference/generated/scipy.linalg.qr_insert.html}.}~\cite{reichel1990algorithm}. If quasi--collinearity is detected, the QR update is skipped. Moreover, at step \(n=0\), the basis is initialized as \(\mat{V}_0 \in \mathbb{R}^{N_h \times 1}\), which consists solely of the normalized initial condition, while if the initial condition is identically zero, we define the basis as equal to the unit vector \([1\, 0 \cdots \,0]^{\tr}\).

Lastly, numerical experiments suggest that it is better to discard the contribution of inner iterates in the basis to compute the solution at the subsequent timestep. This is justified by the observation that iterates are likely to introduce undesired numerical oscillations, jeopardizing the method's performance. Therefore, at time $t_{n+1}$, the initial basis \(\mat{V}_{n+1}^{(0)}\) is computed with a QR update of \(\basis^{(0)}\), to span the subspace identified by the snapshot matrix \([\vec{u}_{n-N_{n+1}+2}, \ldots, \vec{u}_{n+1}]\). 
\begin{algorithm}
\caption{IMEX–RB algorithm}
\label{alg:imexrb}
\begin{algorithmic}[1]
\Require $\vec{y}_0$: initial condition; $N_t$: number of timesteps, $\varepsilon$: absolute stability parameter; $N$: default reduced subspace dimension; $M$: maximal number of inner iterations; $\delta$: lower bound on the reciprocal of the condition number of the augmented basis at each step of QR factorization.

  \State $\vec{u}_0 \gets \vec{y}_0$
  \For{$n = 0,\ldots,N_t-1$}
    \State\label{alg_line:qr}$\vec{V}_n^{(0)}, \mat{R}_n \gets \mathrm{QR}\bigl(\vec{u}_n,\ldots, \vec{u}_{n-\min\{N,n\}+1}; \delta \bigr)$ 
    \For{$k = 0, \ldots, M-1$}\label{alg_line:stop}
      \State \label{alg_line:im}
        Solve $\urb - \Delta t \basis^{(k)\tr} \vec{f}\bigl(t_{n+1},\,\basis^{(k)}\urb+\vec u_n\bigr)=\vec{0}$ 
      \State \label{alg_line:ex}
        $\vec u_{n+1}^{(k)}\gets \vec u_n + \Delta t\vec{f}\bigl(t_{n+1},\basis^{(k)}\urb+\vec u_n\bigr)$
      \State
        $\vec r_{n+1}^{(k)}\gets \left(\eye - \basis^{(k)}\basis^{(k)\tr}\right)\vec u_{n+1}^{(k)}$
      \IIf{$\|\vec r_{n+1}^{(k)}\| / \|\vec u_{n+1}^{(k)}\| < \varepsilon$} \label{alg_line:n+1}\textbf{break} \EndIIf
      \State \label{alg_line:pi}
        $\basis^{(k+1)} \gets \bigl[\basis^{(k)} \,\big|\,
         \vec r_{n+1}^{(k)} / \|\vec r_{n+1}^{(k)}\|\bigr]$
    \EndFor
    \State \label{alg_line:new_sol}
      $\vec u_{n+1}\gets \vec u_{n+1}^{(k)}$
  \EndFor
\end{algorithmic}
\end{algorithm}

\begin{remark}
  The proof of \cref{theo:convergence} only assumes \(\mathcal{V}_n = \spn\{\vec{u}_n\}\) (see \cref{eq:assumption_basis_convergence}). However, the \imex algorithm  enables the reduced basis \(\basis\) to span a larger subspace than \(\spn\{\vec{u}_n\}\), so that
  \( \spn\{\vec u_n\} \subseteq \subspace\). This extra richness is what ensures the absolute stability of the method.
  Nonetheless, we note that our convergence analysis is not affected by the presence of additional columns in the reduced basis matrix.  Indeed, by construction of the \imex algorithm, we have
  \begin{equation}
  \label{eq:remark 3.1}
    \norm{\left(\eye - \dfrac{\vec u_n \vec u_n^{\tr}}{\|\vec u_n\|^2} \right) \vec{u}_{n-i}}
    \to 0
    \quad\text{as }\Delta t\to0, \quad i = 1, \ldots, N_n-1.
  \end{equation}
  Recalling that quasi--collinear modes are discarded during the reduced basis construction process, based on the tolerance $\delta$ in \cref{alg_line:qr} of \cref{alg:imexrb}, 
  \cref{eq:remark 3.1} entails that there always exists $\Delta t ^* > 0$ such that the reduced basis produced by \imex at each timestep is of unit size. Therefore, as the timestep size goes to zero, we naturally fall within the hypothesis \(\mathcal{V}_n = \spn\{\vec{u}_n\}\). 
\end{remark}

\begin{remark}
    In the proof of absolute stability (\cref{theo:abs_stability}), the \(-\varepsilon \mu_{\mathrm{max}} K_2(\mat{T})\) term in the denominator of \cref{eq:denom_abs_stability} imposes an upper bound on the permissible value of \(\varepsilon\) to ensure stability at each time step. This estimate reflects a worst-case scenario in which the dynamics associated with the most numerically unstable eigenmode are not treated implicitly. However, by construction, the \imex method enforces alignment of the reduced basis with the dominant modes of the solution at time \(\vec{u}_{n+1}\). If the mode corresponding to the largest--in--module eigenvalue is captured by the basis, it is automatically included in the implicit step. The actual stability threshold on \(\varepsilon\) may be less restrictive than the derived bound of \cref{eq:bound_eps}. Numerical testing is mandatory to assess its required value.
\end{remark}

\subsection{Computational costs}
We now perform a comprehensive analysis of the computational cost of the \imex algorithm.

Let \(J\in \mathbb{N}^+\) be the maximal number of iterations of the nonlinear solver, used to solve the reduced implicit step at each inner iteration. For simplicity, let us suppose that at each iteration the basis has the maximal size \(\tilde{N} = M + N - 1\). The overall complexity of \textit{one timestep} of \imex is obtained by summing up the costs of all steps of the algorithm:
\begin{itemize}
    \item  \(\mathcal{O}(N_h^2)\) operations for assembling the full--order Jacobian at each timestep. We opt for a quasi--Newton method, to assemble the Jacobian only once per timestep;
    \item  \(\mathcal{O}(M J(N_h^2 + N_h \tilde{N} + \tilde{N}^3))\) operations for the implicit RB step using the quasi--Newton method. This accounts for the evaluation/assemble of \(\vec{f}\) (generally involving matrix--vector products, thus \(\mathcal{O}(N_h^2)\)), its projection onto the reduced subspace (\(\mathcal{O}(N_h \tilde{N})\)) and the solution of a small dense linear system (\(\mathcal{O}(\tilde{N}^3)\)) for up to \(J\) nonlinear solver iterations and up to \(M\) inner iterations;
    \item  \(\mathcal{O}(N_h^2 M)\) operations for the explicit step, which only involves matrix--vector products to compute \(\vec{u}_{n+1}\);
    \item \(\mathcal{O}(M N_h \tilde{N})\) operations for evaluating the stability criterion on \(\vec{u}_{n+1}\);
    \item  \(\mathcal{O}(M N_h \tilde{N})\) operations to enrich the reduced basis matrix through the Gram--Schmidt procedure.
\end{itemize}

Therefore, under the assumption that  $\tilde{N} \ll N_h$, the repeated evaluation of the function \(\vec{f}\), scaling as \(\mathcal{O}(N_h^2)\), can be the bottleneck of the algorithm. 
Moreover, the adoption of a quasi--Newton approach to deal with nonlinearities and the Gram--Schmidt enrichment of the reduced basis entail that the assembly and enrichment of the reduced Jacobian only have a marginal dependency on \(M\). Indeed, once the full-order Jacobian is assembled and projected onto the initial reduced subspace at the beginning of the timestep, which has a cost of \(\mathcal{O}(N_h^2 N_n)\), inner updates of the reduced Jacobian require only \(\mathcal{O}(N_h\tilde N)\) operations per iteration.
To produce a fair comparison of computational times, the same quasi--Newton approach is also employed with the backward Euler method. Nonetheless, the complexity of the latter scales as \(\mathcal{O}(N_h^3)\) per nonlinear iteration. Therefore, computational gains are expected.

Regarding memory storage, at each time step, thanks to the QR update procedure, it is sufficient to store only \(\bm{R}_n\) of size \(N_n\times N_n\), without retaining the full matrix of physical snapshots. Additionally, the basis matrix \(\bm{V}_n\) must be stored, which is enriched during subiterations up to a maximum dimension \(N_h\times \tilde{N}\).

\begin{remark}
In the above analysis, we did not account for the sparsity of the full--order Jacobian matrix, which is in fact of primary importance when estimating computational costs. Indeed, sparse linear systems arising from suitable spatial discretizations of PDEs can often be solved in \(\mathcal{O}(N_h^2)\) operations instead of \(\mathcal{O}(N_h^3)\), using preconditioned iterative methods such as GMRES~\cite{golub2013matrix}. Similarly, favorable sparsity patterns allow full--order matrix--vector products to be performed in less than \(\mathcal{O}(N_h^2)\) operations. As a result, the actual computational gains of IMEX--RB over backward Euler are highly problem-dependent and cannot be predicted a priori. Therefore, they should be investigated numerically.
\end{remark}

When solving nonlinear problems with the RB method, it is common practice to rely on hyper--reduction strategies to further reduce computational costs. In this regard, a widely adopted approach is the Discrete Empirical Interpolation Method (DEIM)~\cite{chaturantabut2010nonlinear}, which enables the construction and efficient use of approximate affine decompositions for the nonlinearities; its extension to matrices (MDEIM) has also been proposed~\cite{negri2015efficient}. More recently, the S--OPT method was introduced in~\cite{lauzon2024sopt} and outperformed DEIM in accuracy. Additionally, if the weak form of the nonlinear term is trilinear --- such as in the Burgers' equation analysed in \cref{subsec: viscous burgers 2D} --- non--intrusive approximate affine decompositions can be derived directly from the reduced basis elements~\cite{pegolotti2021model, tenderini2025model}.
Although these techniques could be implemented within the \imex algorithm, we do not expect them to yield significant computational gains. Indeed, compared to the RB method, \imex does not feature an offline--online splitting paradigm, since it adaptively updates the reduced basis at each timestep. For this reason, we anticipate that any speedup achieved through hyper--reduction in the implicit RB step would be at least offset by the affine components construction. Consequently, we decided not to explore hyper--reduction in this work.

\section{Numerical results}
\label{sec:numerical results}

In this section, we present the numerical experiments we conducted to assess the behavior of the \imex method and we discuss the obtained results. 

In all numerical tests, we consider a spatial computational domain of the form $\Omega = [0, L]^d \subset \mathbb{R}^d \ (d = 2,3)$. The problem is discretized in space by finite differences on a uniform Cartesian grid, made of \(N_i\) nodes for each dimension \( i=1,\ldots,d\). This entails a total of $N_h = K \prod_{i=1}^d N_i$ DOFs\footnote{With a slight abuse of notation, we denote by \(N_h\) both the total number of unknowns and the number of grid points; these may differ when Dirichlet boundary conditions are imposed.}, \(K\) being the number of solution components. Also, we denote by $h_i \coloneqq L / (N_i -1)$ the grid spacing in each dimension. In the following, we set \(h_i = h\), i.e. we employ the same discretization along all dimensions. The time interval $[0,T]$ is instead partitioned into \(N_t\) uniform sub--intervals of size \(\Delta t \coloneqq T/N_t\).

The performances of \imex are expressed in terms of its errors and computational times, and are compared against those of the backward Euler (BE) method, which serves as a baseline. Regarding the error computation, we always assume the existence of an exact analytical solution $\vec{u}_{\mathrm{ex}}: \Omega\times[0, T]\to\mathbb{R}^K$, which is evaluated at the grid points and whose components we denote by $(\vec{u}_{\mathrm{ex}})_k, \ k = 1, \ldots, K$. Let $\mathcal{I}=\{1,\dots,N_1\} \times \cdots \times \{1, \ldots, N_d\}$ denote the set of grid--point indices, so that \(\vec{i}=(i_1,\dots,i_d)\in\mathcal I\); let  \(\vec{x}_{\vec{i}}=(x_{1,i_1},\ldots,x_{d,i_d})\) be the grid point associated with the multi--index $\vec{i}$. We define \((\vec{u}_n)_{\vec{i},k}\) as the approximation of \(\bigl(\vec{u}_{\mathrm{ex}}(\vec{x}_{\vec{i}},t_n)\bigr)_k\) yielded by the numerical method of choice. Then, following \cite[p.~252]{Leveque},  the 2--norm relative error in space at time $t_n$ is given by 
\begin{equation}\label{eq:error_d}
e_{r,n}
\coloneqq
\left( \sum_{k=1}^K  \frac{\sum\limits_{\vec{i}\in\mathcal I} h^d\,
    \lvert (\vec{u}_n)_{\vec{i},k}-\bigl(\vec{u}_{\mathrm{ex}}(\vec{x}_{\vec{i}},t_n)\bigr)_k\rvert^2}
     {\sum\limits_{\vec{i}\in\mathcal I} h^d\,
    \lvert\bigl(\vec{u}_{\mathrm{ex}}(\vec{x}_{\vec{i}},t_n)\bigr)_k\rvert^2} \right)^{1/2}  =
  \left( \sum_{k=1}^K  \frac{\norm{\vec{e}_{k,n}}^2_2}{\norm{(\vec{u}_{\mathrm{ex}}(\cdot, t_n))_k}^2_2} \right)^{1/2}.
\end{equation}
Ultimately, the accuracies of \imex and BE are quantified by the following aggregate error indicator, which takes into account the behavior of the method over the entire time interval:
\begin{equation}
\label{eq:error_int_time}
\bar{e}_{r}
\coloneqq \left( \sum_{k=1}^K 
\frac{\sum\limits_{n=1}^{N_t} \Delta t \norm{\vec{e}_{k,n}}_2^2}
{\sum\limits_{n=1}^{N_t} \Delta t \norm{(\vec{u}_{\mathrm{ex}}(\cdot, t_n))_k}_2^2 } \right)^{1/2}.
\end{equation}
Advective terms are approximated using second--order centered finite differences, while diffusive terms leverage second--order five--point stencil formulations. For the Burgers' equation (see \cref{subsec: viscous burgers 2D}), the convective term is also treated with second--order centered finite differences. Inhomogeneous Dirichlet boundary conditions are strongly enforced, leveraging lifting techniques. In particular, the numerical solution is expressed as the sum of an unknown component, which evaluates to zero at the Dirichlet boundaries, and a known lifting component, defined to match the Dirichlet data. Notably, this strategy enables looking for solutions that belong to a vector space and inherently configures the low--dimensional space $\mathcal{V}_n$ of \imex as a vector space itself.

The large and sparse linear systems, arising from the application of BE at each discrete timestep, are solved with the GMRES method~\cite{saad1986gmres}, preconditioned through incomplete LU factorization, with a drop tolerance of \(5\cdot10^{-3}\). 
The small dense linear systems associated with the use of \imex are instead handled using the direct solver implemented in the \href{https://docs.scipy.org/doc/scipy/reference/generated/scipy.linalg.solve.html}\emph{SciPy} library.

All numerical simulations are performed on the \emph{Jed} cluster of the \emph{Scientific IT and Application Support (SCITAS)}\footnote{\url{https://www.epfl.ch/research/facilities/scitas/hardware/}.} at EPFL. All runs are executed sequentially. y. To ensure sufficient memory per job, we request 10 CPUs for each 2D test
and 20 CPUs for the 3D test. 
\subsection{Advection--diffusion equation in 2D} \label{subsec: advection diffusion 2D}

\begin{figure}[t!]
\centering
\includegraphics[width=.99\textwidth]{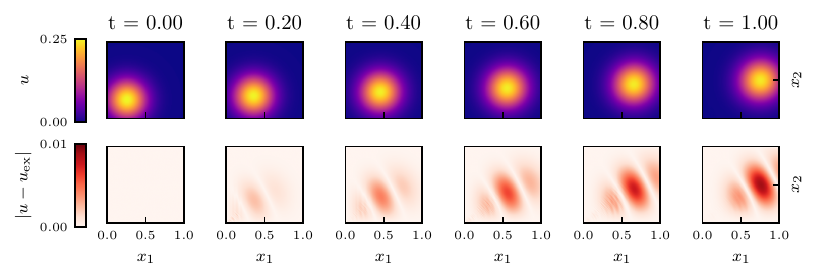}
\caption{Numerical solution to \cref{eq: advection diffusion 2D}, produced by \imex at five equispaced timesteps, and corresponding absolute error with respect to the exact solution. The results were obtained for $N_1=N_2=101$, $N_t=100$, $\varepsilon = 2 \cdot 10^{-3}$, $N=10$ and $M=100$.}
\label{fig: advection-diffusion solution}
\end{figure}

We consider the following two--dimensional linear advection--diffusion problem:
\begin{equation}
\label{eq: advection diffusion 2D}
\begin{cases}
\dfrac{\partial u}{\partial t} + c_x \dfrac{\partial u}{\partial x} + c_y \dfrac{\partial u}{\partial y} - \mu \Delta u = f,
  & \text{in} \ \Omega \times (0, T]~, \\[1em]
u = u_0, 
  & \text{on} \ \Omega \times \{0\}~, \\[0.5em]
u = g, 
  & \text{on} \  \partial \Omega \times (0, T]~,
\end{cases}
\end{equation}
where $\Omega = [0,1]^2$, $\partial\Omega \subset \mathbb{R}^2$ denotes its boundary, and $T=1$. We consider a constant diffusion parameter $\mu = 0.005$ and advection velocity $\vec{c} = [c_x, c_y] = [0.5, 0.25]$. The forcing term $f : \Omega \times (0,T] \to \mathbb{R}$, the initial condition $u_0: \Omega \to \mathbb{R}$, and the Dirichlet datum $g : \partial\Omega \times (0, T] \to \mathbb{R}$ are defined so that the exact solution to \cref{eq: advection diffusion 2D} writes as follows:
\begin{equation}\label{eq:ex_sol_AD}
u_{\text{ex}}(\vec{x}, t) = U \exp\!\left(
-\frac{\Vert \vec{x} - \vec{x}_0 - \vec{c} t \Vert_2^2}{\sigma^2 + \mu t}
\right)~.
\end{equation}
We note that $\sigma > 0$ regulates the amplitude of the peak for the initial condition, $U>0$ is the peak solution value at $t=0$, and $\vec{x}_0 \in \Omega$ is the peak location at $t=0$. In all the tests, we set $\sigma = 0.25$, $U = 0.25$, and $\vec{x}_0 = [0.25, 0.25]$. From a qualitative standpoint, we remark that the solution consists of a traveling Gaussian blob, driven by the advection field, whose diffusion is fully counterbalanced by the external forcing term. \Cref{fig: advection-diffusion solution} reports the numerical solution obtained with \imex --- for a specific spatio--temporal discretization of the domain and choice of the hyperparameters $\varepsilon, N, M$ --- at six equispaced timesteps, and the corresponding absolute errors with respect to the exact solution. 

\begin{figure}[t!]
  \centering
  \subfloat[]{%
    \label{fig:advDiff2D_eps}%
    \includegraphics[width=0.48\textwidth]{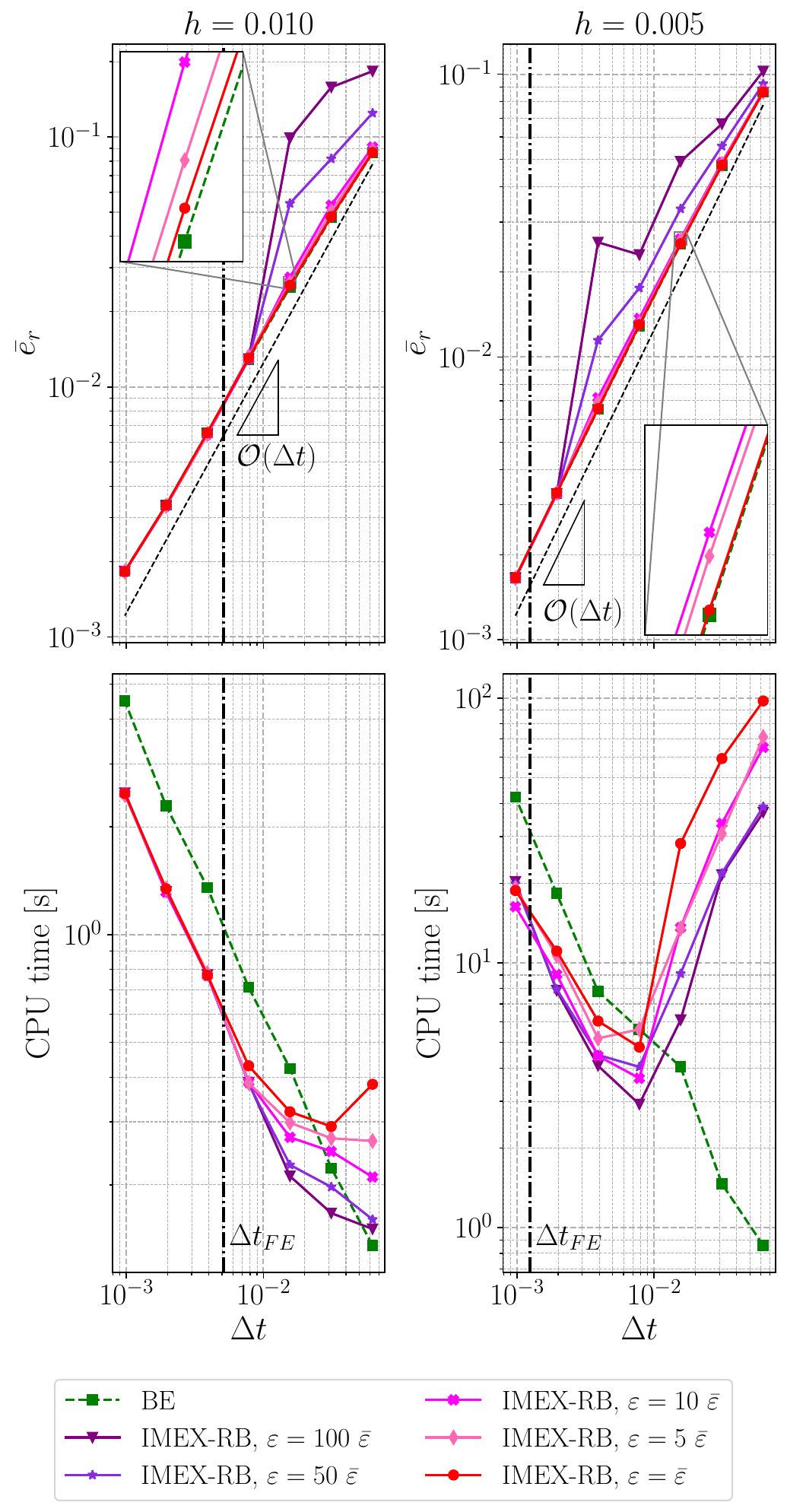}%
  }\hfill
  \subfloat[]{%
    \label{fig:advDiff2D_N}%
    \includegraphics[width=0.48\textwidth]{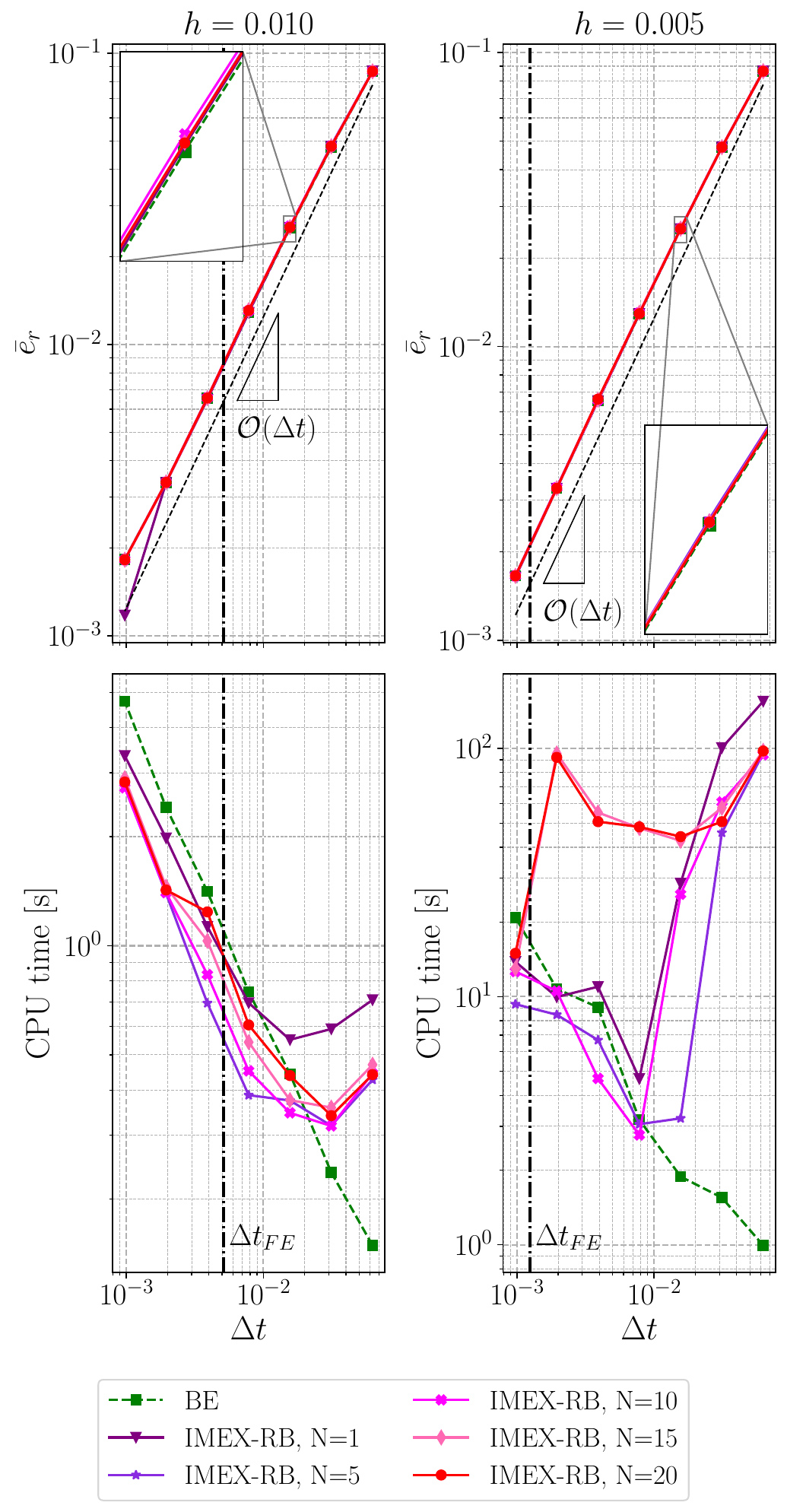}%
  }
  \caption{Backward Euler (BE) and \imex results on the 2D advection--diffusion problem: convergence analysis and CPU times (in seconds).
    \protect\subref{fig:advDiff2D_eps} Results varying the stability tolerance~$\varepsilon$
     (\textit{left}: $h=0.01$, \textit{right}: $h=0.005$;  \imex employs $N=10,\,M=100$).
    \protect\subref{fig:advDiff2D_N} Results varying the minimal reduced basis size~$N$
    (\textit{left}: $h=0.01$, \textit{right}: $h=0.005$;
    \imex employs $\varepsilon = \bar{\varepsilon}$, $M=100$).
    CPU times are averaged over ten runs.}
  \label{fig: advDiff2D test vs epsilon}
\end{figure}

We start by empirically investigating the role of the stability tolerance $\varepsilon$. \Cref{fig:advDiff2D_eps} reports the relative errors --- computed according to \cref{eq:error_int_time} --- and the computational times, associated with BE and \imex for five different values of $\varepsilon$, seven different timestep sizes $\Delta t_i = 2^{-i}$ for $i = 4, \ldots, 10$, and two different mesh sizes $h = 0.01$ and $h = 0.005$ (\(N_i = 101,\, 201\)). As reported in \cref{subsec: absolute stability of imex-rb}, the values of $\varepsilon$ are chosen as multiples of $\bar{\varepsilon} = (K_2(\mat{A}))^{-1}$.
Specifically, we have that $\bar{\varepsilon} \approx 2.1 \cdot 10^{-3}$ for $h = 0.01$ and $\bar{\varepsilon}\approx 5.3 \cdot 10^{-4}$ for $h = 0.005$. We also report the timestep size $\Delta t_{FE}$ below which the forward Euler method is stable. 
We begin by highlighting a few key observations.
Firstly, if the timestep size $\Delta t$ is sufficiently small, the errors of \imex are basically identical to those of BE and feature a linear convergence with respect to $\Delta t$. This result provides an empirical validation of \cref{theo:convergence}. 
Secondly, if $\varepsilon$ is sufficiently small (say, $\varepsilon \lesssim 10 \ \bar{\varepsilon}$), \imex produces accurate results even for timestep values that lie well above the critical value for forward Euler. This result confirms that, upon suitable choices of the hyperparameters, the \imex scheme is stable, and also suggests that the bound of \cref{eq:bound_eps} might be too restrictive in practice.
Additionally, we note that more stringent constraints on $\varepsilon$ are required as $h$ decreases, to obtain accurate solution approximations. In particular, we observe that, with both space discretizations, the errors of \imex diverge from those of BE as soon as $\Delta t > \Delta t_{FE}$, if $\varepsilon$ is too large. This behavior could be easily inferred from the theory. Indeed, there exists an inverse proportionality relation between $C_2$ --- the stability tolerance upper bound in \cref{eq: stability bound} --- and the condition number of the system matrix, and the latter increases as the spatial grid gets refined. 
\begin{figure}[t!]
  \centering
  \subfloat[
    \label{fig:advDiff2D_iters_eps}]{%
    \includegraphics[width=0.48\textwidth]{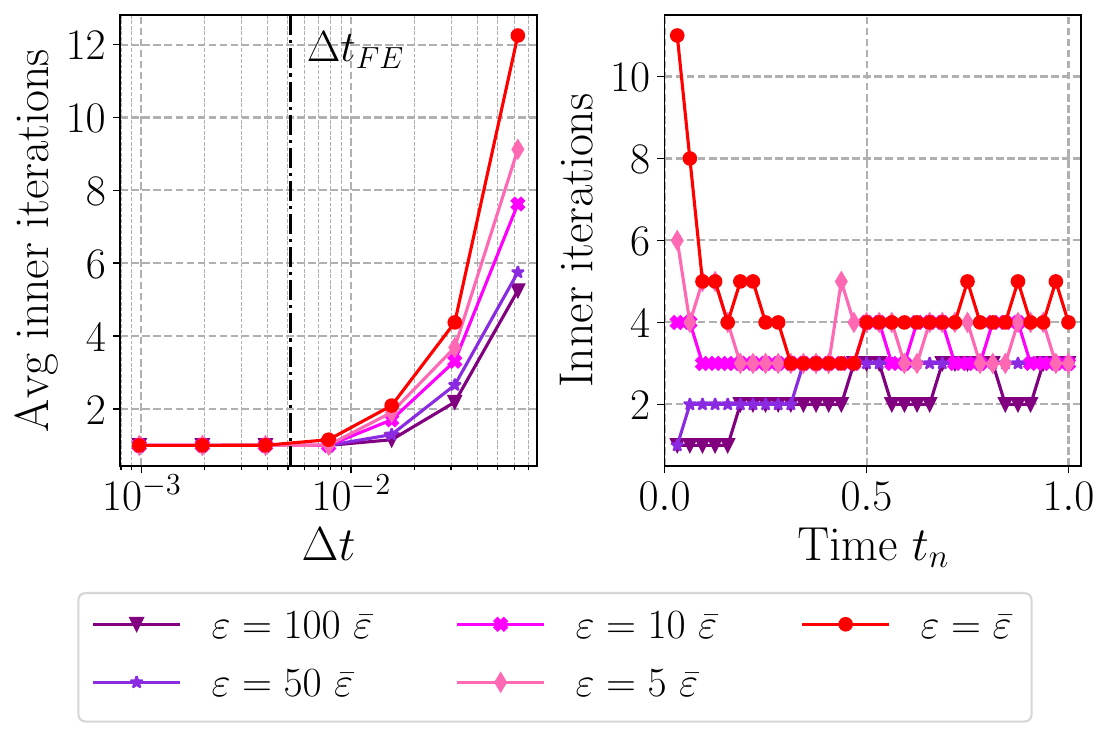}%
  }\hfill
  \subfloat[
    \label{fig:advDiff2D_iters_N}]{%
    \includegraphics[width=0.48\textwidth]{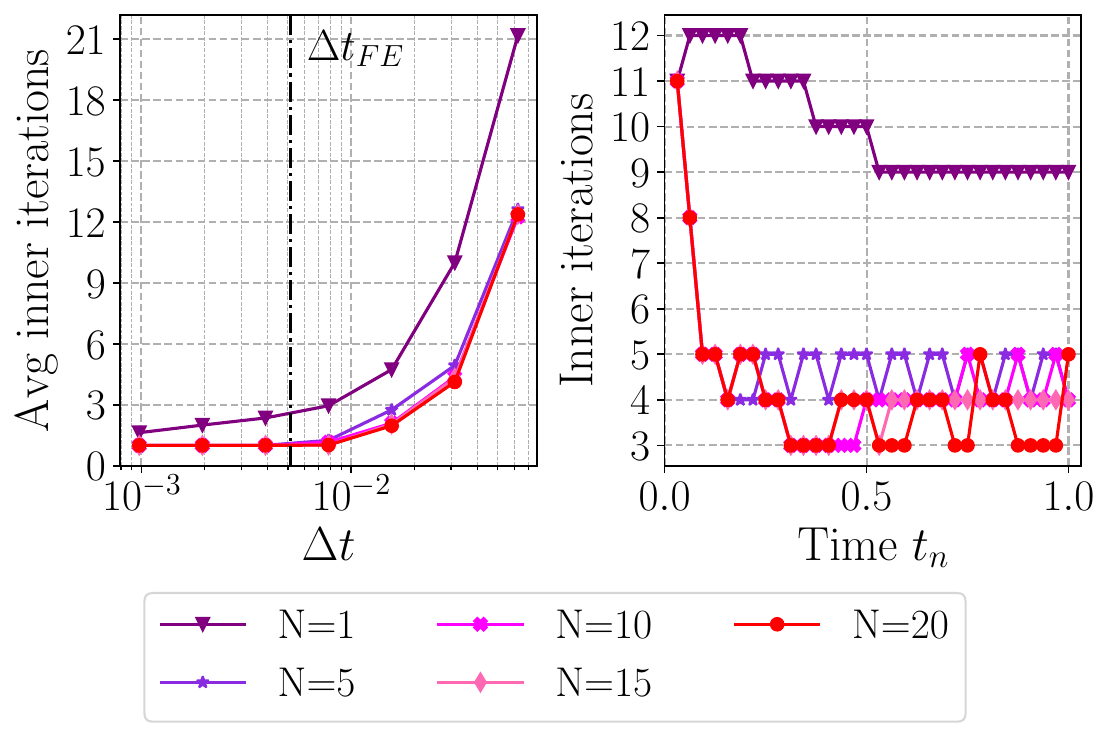}%
  }
  \caption{Number of inner iterations performed by the \imex method on the 2D advection--diffusion problem. In particular, we investigate how the average number of inner iterations varies for different values of \(\Delta t\), and the trend of iterations over time for \(\Delta t = 1/32\). \protect\subref{fig:advDiff2D_iters_eps} Results varying the values of the stability tolerance $\varepsilon$. \protect\subref{fig:advDiff2D_iters_N} Results varying the reduced basis initial size \(N\). All tests were conducted setting $N_1=N_2=101$ and a maximal number of inner iterations $M=100$.}
  \label{fig:advDiff2D_subiterates_vs_epsilon_and_N}
\end{figure}

Concerning computational efficiency, we observe that \imex outperforms BE for sufficiently small values of $\Delta t$, while BE is faster than \imex for large timestep values. 
In this regard, we underline that CPU runtimes of both algorithms strongly depend on the problem at hand, on the hardware specifics, and on the software implementation. As such, the reported times only offer indicative trends.
Remarkably, the computational gains of \imex over BE are not limited to the case $\Delta t < \Delta t_{FE}$, where forward Euler is stable and hence configures as the first--order method of choice. Indeed, for both $h = 0.01$ and $h = 0.005$ and taking as reference the case $\varepsilon = \bar{\varepsilon}$, \imex is faster than BE until $\Delta t \lesssim 5 \ \Delta t_{FE}$. For instance, setting $h = 0.005$ and $\Delta t = 1 / 128 \approx 6 \Delta t_{FE}$, the average runtimes of BE and \imex are, respectively $4.39 \ s$ and $3.07 \ s$ ($-30 \%$). The explosion of the computational times of \imex at large values of $\Delta t$ can be justified by taking into account the number of inner iterations of the algorithm, which are shown in \cref{fig:advDiff2D_iters_eps}. As larger timestep values are employed, more iterations are required to satisfy the absolute stability constraint, which demands more computational resources. The number of inner iterations increases as smaller values of $\varepsilon$ are selected. From \cref{fig:advDiff2D_iters_eps}, we note that the largest number of iterations occurs at the initial timesteps, where the solution history, and hence the baseline reduced basis dimension $N$, is small. As time progresses, only a few inner cycles ($3$ to $5$ for $\varepsilon=\bar{\varepsilon}$) are instead necessary to guarantee stability. 

The performances of \imex considering different values of $N$ are summarized in \cref{fig:advDiff2D_N}, where we report both relative errors and computational times, for timestep values $\Delta t = 2^{-i}, \ i=4,\ldots 10$, and mesh sizes $h = 0.01$ and $h =0.005$. For these numerical experiments, we set $\varepsilon = \bar{\varepsilon}$. Once again, we observe a linear decay of the error with respect to $\Delta t$ and, for well--calibrated values of $N$, computational gains with respect to BE, if $\Delta t \lesssim 5 \ \Delta t_{FE}$. In particular, we remark that the discrepancy in performance for different values of $N$ cannot be detected from the errors, but rather from the CPU runtimes. For \(N \geq 15\), significantly higher computational times are observed for \imex, due to the repeated orthogonalizations required by the QR update procedure. This issue can be mitigated by relaxing the tolerance criterion used to skip the QR update. Furthermore, since the method allows for adaptive augmentation of the reduced basis via inner iterations, a small baseline value of \(N\) may lead to substantial growth of the low--dimensional subspace, thereby increasing the overall computational cost. Therefore, the performances can be optimized by choosing a value of $N$ that is large enough to guarantee a good approximation of the solution manifold at each timestep, and small enough to save computational resources. Unfortunately, the optimal value of $N$ is highly problem dependent, and \emph{a priori} estimates are difficult to derive. For the test case at hand, as also demonstrated by the inner iterations results reported in \cref{fig:advDiff2D_iters_N}, the best trade--off is offered by $N=5$. For $N=1$, many inner iterations are needed at all timesteps, which entails relevant costs for large timestep values; for $N=15$, the reduced basis size is unnecessarily big, and \imex is never able to outperform BE.

\subsection{Viscous Burgers' equation in 2D} 
\label{subsec: viscous burgers 2D}
To assess the performance of \imex on a nonlinear PDE, we consider the following 2D viscous Burgers' problem: find \(\vec{u} = [u_1, u_2]^{\tr}\) such that
\begin{equation}
\label{eq:2D_burgers}
\begin{cases}
    \dfrac{\partial \vec{u}}{\partial t} + (\vec{u} \cdot \nabla) \vec{u} - \nu \Delta \vec{u} = 0,  \quad & \text{in} \  \Omega \times (0, T], \\
    \vec{u} = \vec{u}_0, & \text{on} \  \Omega \times \{0\}, \\
    \vec{u} = \vec{g}, \quad &  \text{on} \  \partial\Omega \times (0, T], \\
\end{cases}
\end{equation}
where $\Omega = [0,1]^2$, $\partial\Omega \subset \mathbb{R}^2$ denotes its boundary, and $T=1$. The quantity $\nu \in \mathbb{R}^+$ denotes the kinematic viscosity; in all tests, we set $\nu = 10^{-2}$.
In particular, we consider a well--known benchmark problem \cite{BAHADIR2003131, YANG2021510}, whose analytical solution is given by
\begin{equation*}
\label{eq:exact_sol_2D_Burgers}
\vec u_{\mathrm{ex}}(\vec x,t)
= \frac{3}{4}
  - \frac{1}{4}
    \begin{bmatrix}
      \left(1 + \exp\left(\dfrac{-4x_1 + 4x_2 - t}{32\nu}\right)\right)^{-1}
      \\[2ex]
      -\left(1 + \exp\left(\dfrac{-4x_1 + 4x_2 - t}{32\nu}\right)\right)^{-1}
    \end{bmatrix}.
\end{equation*}
The initial condition $\vec{u}_0: \Omega \to \mathbb{R}^2$, and the Dirichlet datum $\vec{g} : \partial\Omega \times (0, T] \to \mathbb{R}^2$ are inferred from the exact solution.

At each time step, for both BE and \imex, a nonlinear system is solved using a quasi--Newton method. The Jacobian matrix of the PDE residual is assembled once at the beginning of the time step, evaluated at \((t_{n+1}, \vec{u}_n)\), and kept fixed throughout the iterations; in this context, \(\vec{u}_n\) acts as a first--order extrapolation of \(\vec{u}_{n+1}\). The iterations are terminated when the Euclidean norm of the update falls below \(10^{-3} \cdot \Delta t\), or at most after $J=100$ loops.   
The GMRES iterations used in the nonlinear solver are stopped once the relative residual measured with respect to the norm of the right--hand side vector is smaller than \(10^{-10}\).

Unless differently specified, we set $N_1 = N_2 = 101$ --- for this choice, FE is stable for $N_t \gtrsim 400$ --- and a number of timesteps \(N_t = 40\). For \imex, we consider a minimal reduced basis size $N = 10$ and a maximal number of inner iterations $M=100$. Since the nonlinear nature of the problem prevents computing a reliable estimate of the stability tolerance \(\varepsilon\) \textit{a priori}, we empirically derive an educated guess by running \emph{ad hoc} preliminary tests. 

We begin by investigating the absolute stability of \imex. We fix both the space and time discretization, and we let \(\varepsilon = 10^{-i},\, i = 2,3,4,5\). 
\begin{figure}[t]
  \centering
  \subfloat[]{%
    \label{fig:burgers_energynorm}%
    \includegraphics[width=0.65\linewidth]{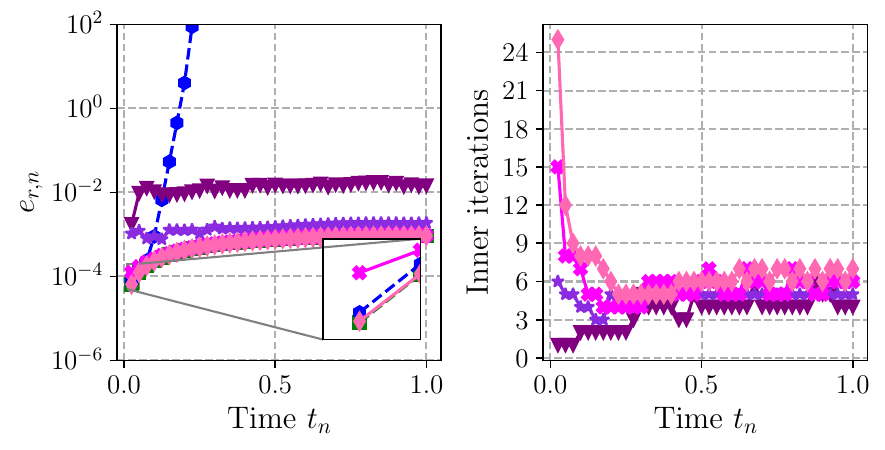}%
  }\hfill
  \subfloat[]{%
    \label{fig:burgers_convergence}%
    \includegraphics[width=0.25\linewidth]{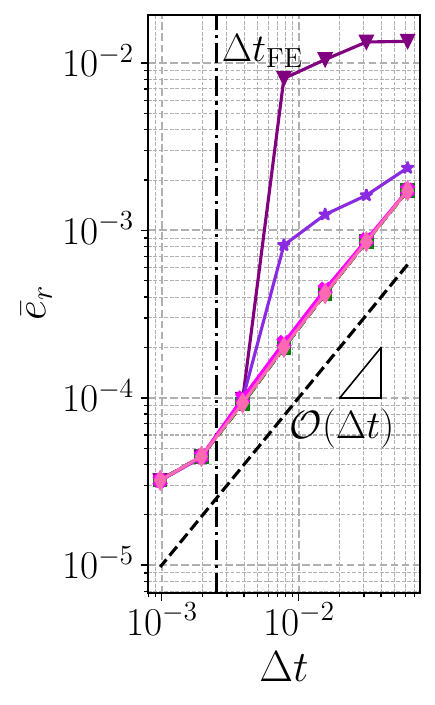}%
  }  \par
  \centering
  \subfloat{
    \includegraphics[width=0.55\linewidth]{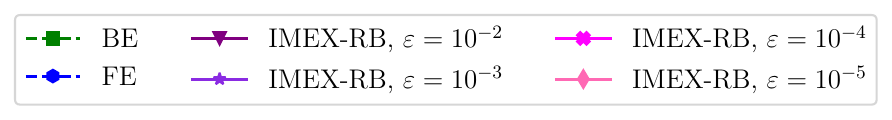}%
  }
  
  \caption{Backward Euler (BE) and \imex results on the 2D viscous Burgers' equation. 
    \protect\subref{fig:burgers_energynorm} Relative error \(e_{r,n}\) over time. Left: results for \imex, varying \(\varepsilon\), BE and forward Euler (FE), setting \(N_t = 40\), \(N=10\), \(M=100\). Right: number of inner iterations performed by \imex.
    \protect\subref{fig:burgers_convergence} Convergence analysis of \imex on both solution components. FE is only stable for values of \(\Delta t\) to the left of the dashed line \(\Delta t_{FE}\).
  }
  \label{fig:burgers}
\end{figure}
\Cref{fig:burgers_energynorm} shows both the relative error trend over time and the number of inner iterations performed by the \imex scheme.
Firstly, we highlight that FE is numerically unstable for the selected value of \(N_t\), as its error quickly explodes after a few timesteps. Secondly, if \(\varepsilon \geq 10^{-3}\), \(e_{r,n}\) shows a jump with respect to that of BE already after the first timestep. This suggests that if the value of \(\varepsilon\) is too loose, some numerically unstable modes are immediately amplified by the explicit step, thus deteriorating accuracy. However, for the subsequent timesteps, setting \(\varepsilon = 10^{-2}\) suffices to keep the error bounded over the entire time interval, and the error discrepancy compared to BE is only due to what occurs at the initial timesteps. These considerations are also reflected in the observed number of inner iterations, which is significantly affected by the choice of \(\varepsilon\) only up to \(t_n \approx 0.2\). For the following timesteps, all choices of \(\varepsilon\) lead to almost the same number of iterations, which indicates that the subspace quickly captures all the unstable modes necessary to meet the inequality of \cref{eq:absolute_stab_cond}. Finally, this test case is exploited to define a value of $\varepsilon$ to be used in the following experiments. Based on \cref{fig:burgers_energynorm} and aiming at an accuracy comparable to that of BE, we conclude that \(\varepsilon=10^{-4}\) suffices to guarantee stability, without compromising accuracy.
\begin{figure}[t]
    \centering
\includegraphics[width=0.9\linewidth]{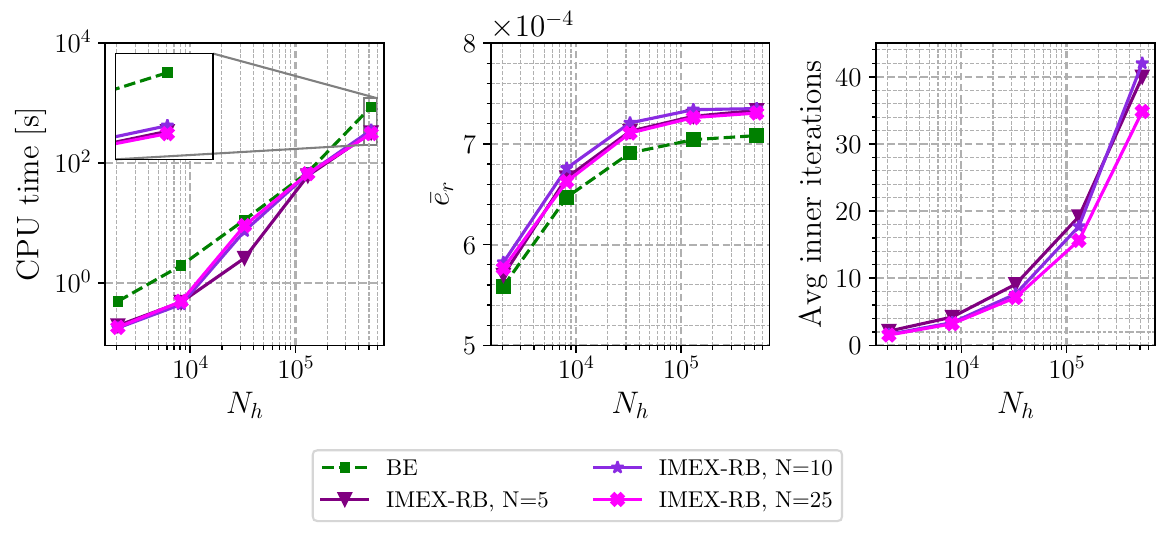}
    \caption{Backward Euler (BE) and \imex performances for different spatial discretizations, on the 2D viscous Burgers' equation. In particular, we show CPU times (\textit{left}) and errors (\textit{center}) for BE and \imex, considering different values of \(N\). We also report the average number of inner iterations performed by \imex (\textit{right}).  Here, we set \(N_t = 40\) and \(\varepsilon=10^{-4}\).}
    \label{fig:cpu_times_burgers}
\end{figure}

As a second numerical test, we perform a convergence study, keeping the spatial grid fixed (\(h = 0.01\)) and varying the timestep sizes as \(\Delta t =2^{-i}\) for \(i=4,\dots,10\). Also, we vary \(\varepsilon = 10^{-i}, \, i=2,\ldots,5\). As before, we fix \(N=10\) and \(M=100\).  At each timestep, the resulting nonlinear system features approximately \(2\cdot10^{4}\) DOFs.  For each choice of \(\Delta t\), we compute the error according to the norm defined in \cref{eq:error_int_time}.  The results are displayed in \cref{fig:burgers_convergence}. In agreement with \cref{theo:convergence}, \imex is convergent of order 1 in time. It can be noted that already for \(\varepsilon=10^{-4}\), the stability of \imex is reflected in the achieved accuracy, and for values of \(\Delta t\) above \(\Delta t_{\mathrm{FE}}\). On the contrary, for larger values of \(\varepsilon\), despite meeting the bound of \cref{eq:absolute_stab_cond}, absolute stability is not guaranteed, due to the premature termination of inner iterations.

Finally, in \cref{fig:cpu_times_burgers}, we compare the CPU times of BE and \imex, considering different spatial resolutions. Specifically, we set \(N_{1} = N_{2} = 2^{i}\) for \(i = 5, 6, \ldots, 9\), which results in up to \(N_{h} = 2N_{1}N_{2} \approx 5 \times 10^{5}\) DOFs. We fix \(N_t = 40\), \(M = 100\), and \(\varepsilon = 10^{-4}\); also, simulations are performed for \(N = 5, 10, 25\). For each problem size, CPU times are averaged over five runs. In addition to CPU times, we report the error as defined in \cref{eq:error_int_time} and the average number of inner iterations performed by \imex. The CPU times of \imex are lower than those of BE for practically all values of \(N_h\), and suggest that the computational gain may increase for larger values of \(N_h\). This behavior aligns with the goal we sought when designing \imex: for large--scale problems, \imex is expected to be more computationally efficient than BE. At the same time, both schemes achieve comparable accuracy for all values of \(N_h\), even for \(N = 5\). As predictable, the average number of inner iterations grows with \(N_h\), since a larger full--order problem requires a higher--dimensional reduced subspace to capture the unstable modes. Nonetheless, the number of required iterations remains sufficiently low to yield computational advantages over BE.

Overall, the \imex scheme showcases superior computational efficiency compared to BE when applied to the problem in \cref{eq:2D_burgers}, while achieving comparable accuracy levels. Additionally, although increasing the reduced basis size \(N\) generally reduces the average number of inner iterations, especially for refined spatial grids, we observe that setting \(N = 5\) already yields satisfactory results in the considered scenario.

\subsection{Advection--diffusion equation in 3D}
\label{subsec: advection diffusion 3D}
We test \imex on the 3D counterpart of the problem of \cref{eq: advection diffusion 2D}. The selected exact solution in this scenario is again given by \cref{eq:ex_sol_AD}, where this time \(\vec{x} = [x_1, x_2, x_3]^{\tr}\). We solve in the domain \(\Omega = [0,1]^3\), and we select \(\vec{c} = [0.5, 0.25, 0.25]^{\tr}\), \(\vec{x}_0 = [0.25, 0.25, 0.25]^{\tr}\), \(\sigma = 0.25\), \(\mu = 10^{-2}\). For this test case, we are interested in studying how the \imex scheme behaves when increasing the dimension of the physical problem and thus the number of DOFs.

We begin by studying its convergence and the corresponding CPU times, for two different mesh sizes, \(h=0.02\) and \(h= 0.012\) (\(N_i = 51\) and \(N_i = 81\), \(i=1,2, 3\)), and varying the value of \(\varepsilon\) with respect to \(\bar{\varepsilon} = K_2(\mat{A})^{\inv}\). In particular, we vary \(\Delta t = 2^{-i}, \, i=4, \ldots, 10\); we obtain \(\bar{\varepsilon} \approx 3.28 \cdot 10^{-3}\) and \(\bar{\varepsilon} \approx 1.31 \cdot 10^{-3}\) for the considered mesh sizes\footnote{To speed up the computations, the condition number of \(\mat{A}\) is obtained as ratio of singular values computed up to a tolerance of \(10^{-2}\).}. \Cref{fig:advDiff3D_eps} displays the results.
\begin{figure}[t!]
  \centering
  \subfloat[]{%
    \label{fig:advDiff3D_eps}%
    \includegraphics[width=0.48\textwidth]{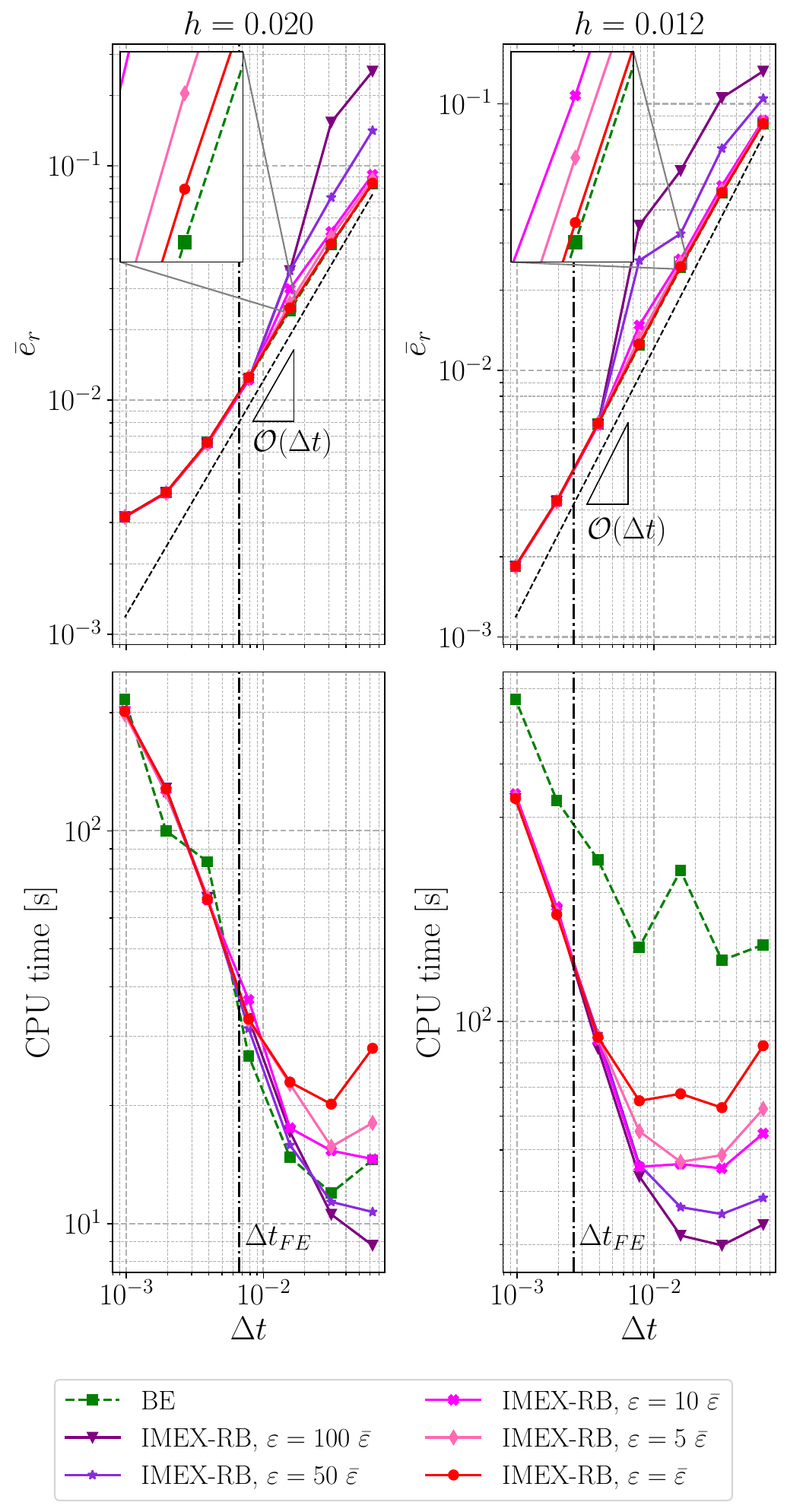}%
  }\hfill
  \subfloat[]{%
    \label{fig:advDiff3D_N}%
    \includegraphics[width=0.48\textwidth]{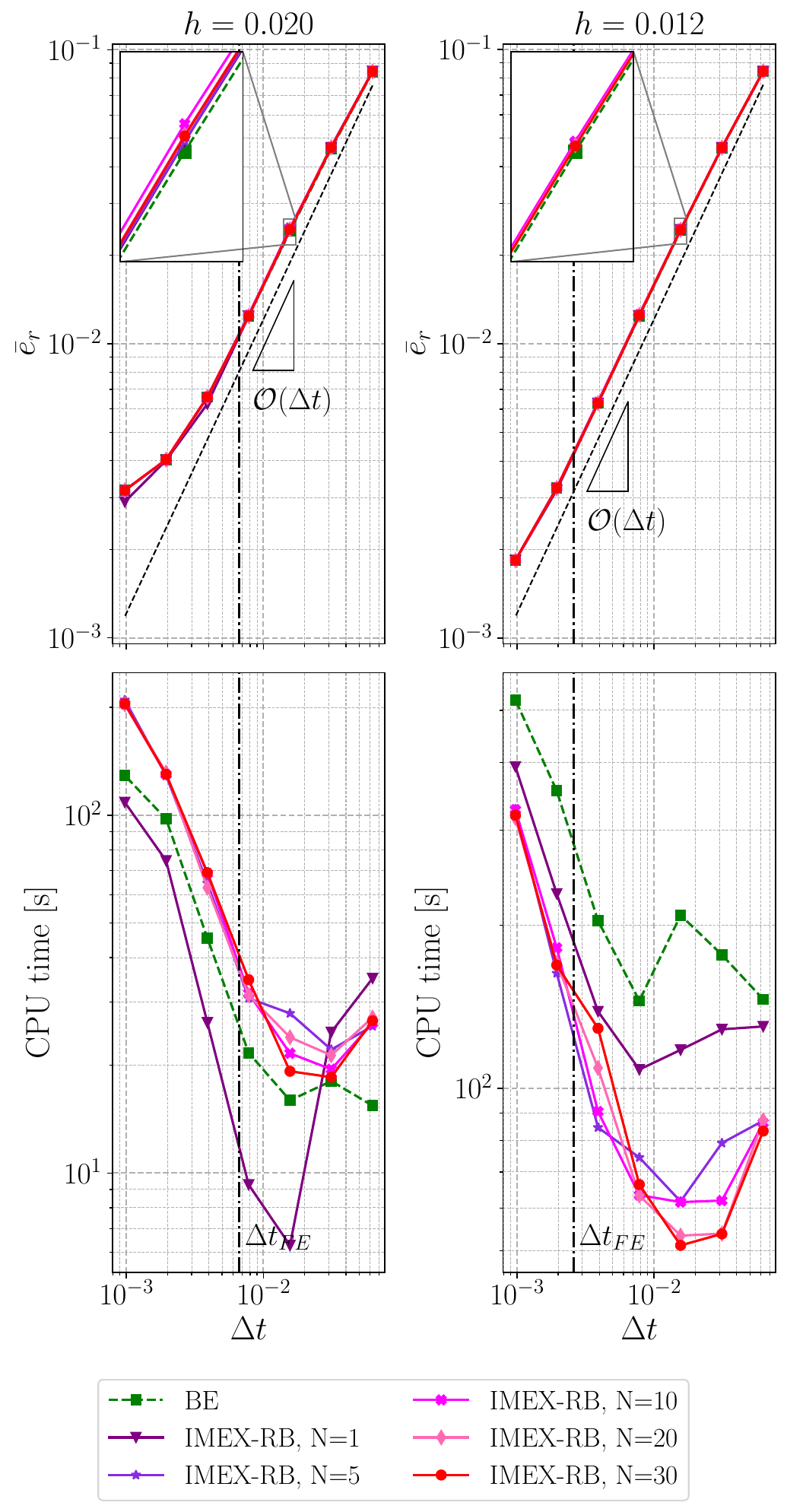}%
  }
  \caption{Backward Euler (BE) and \imex results on the 3D advection–diffusion problem: convergence analysis and CPU times (in seconds).
    \protect\subref{fig:advDiff3D_eps} Results varying the stability tolerance~$\varepsilon$
     (\textit{left}: $h=0.02$, \textit{right}: $h=0.012$; \imex employs $N=10,\,M=100$).
    \protect\subref{fig:advDiff3D_N} Results varying the minimal reduced basis size~$N$
    (\textit{left}: $h=0.02$, \textit{right}: $h=0.012$;
    \imex employs $\varepsilon = \bar{\varepsilon}$, $M=100$).
    To save computational resources, the CPU times in this scenario are not averaged.}
  \label{fig:advDiff3d_conv}
\end{figure}
With respect to the same test conducted on the 2D problem (see  \cref{fig: advDiff2D test vs epsilon}), we remark a similar behavior in convergence, for both values of \(h\) considered: \(\varepsilon \lesssim 10\,\bar{\varepsilon}\) suffices to attain an error comparable to that of BE. Besides, the two--step approach of \imex makes it possible to achieve numerical stability for values of \(\Delta t\) well above the critical value for forward Euler \(\Delta t_{\mathrm{FE}}\). Moreover,  \cref{theo:convergence} is once more validated, as we observe that the error behaves as \(\mathcal{O}(\Delta t)\). Regarding CPU times, while keeping in mind the limits of the measured times, we highlight that for \(h = 0.02\), \imex is less efficient than BE. As the problem grows in size, the \(\mathcal{O}(N_h^3)\) cost paid by BE makes \imex more efficient: for \(\Delta t > \Delta t_{FE}\), and \(h = 0.012\), \imex requires half as much the computational time of the BE scheme. As already observed in the 2D scenario, a sweet spot enables \imex to achieve the fastest computational times for values of \(\Delta t > \Delta t_{FE}\), but not too large. Indeed, for the largest values of \(\Delta t\), \imex performs too many inner iterations, thus hindering computational efficiency.

Further insights are provided by \cref{fig:advDiff3D_iters_eps}, which analyzes the scenario \(h = 0.012\) in more detail. Here, the smaller the imposed value of \(\varepsilon\), the larger the number of inner iterations that must be performed to meet the absolute stability criterion of \cref{eq:absolute_stab_cond}. This explains the increasing CPU times for decreasing values of \(\varepsilon\) for \(\Delta t > \Delta t_{FE}\). The trend of average number of inner iterations varying \(\Delta t\) also shows that, for \(\Delta t \lesssim \Delta t_{FE}\), one inner iteration suffices to guarantee absolute stability for all different thresholds \(\varepsilon\). This is because, as \(\Delta t \to 0\), \(\vec{u}_n \to \vec{u}_{n+1}\); the initial subspace at timestep \(n+1\) already almost contains \(\vec{u}_{n+1}\). Hence, the inequality of \cref{eq:absolute_stab_cond} is more easily satisfied. 
\begin{figure}[t!]
  \centering
  \subfloat[]{%
    \label{fig:advDiff3D_iters_eps}%
    \includegraphics[width=0.48\textwidth]{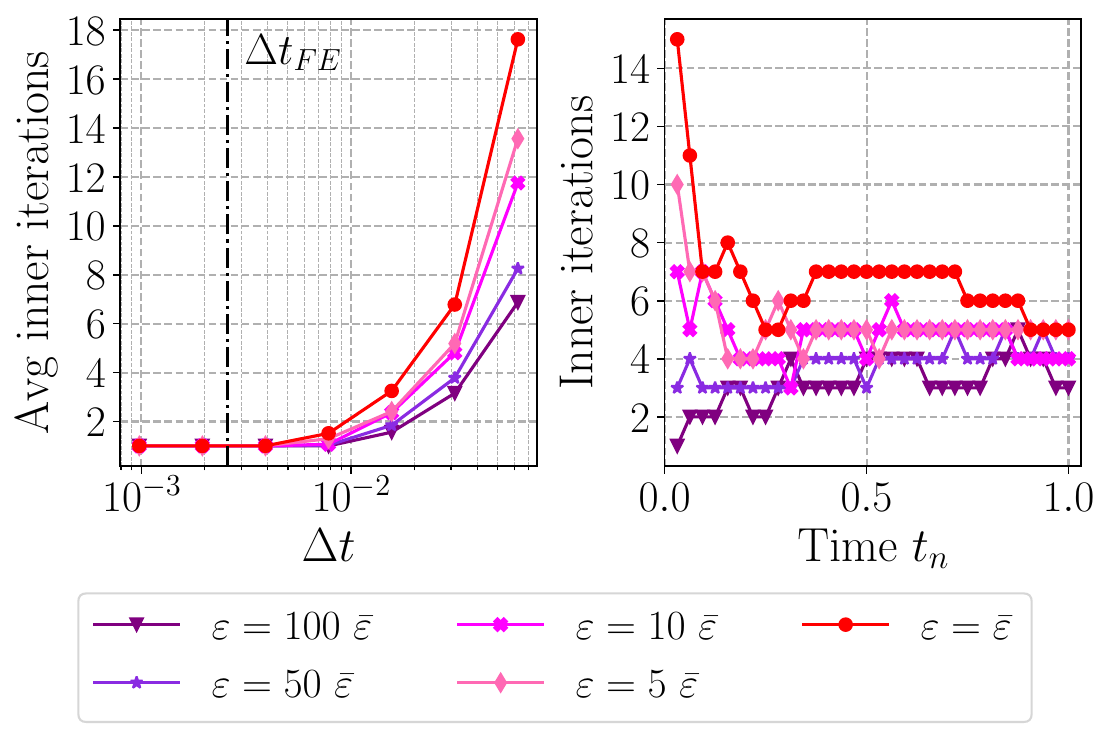}%
  }\hfill
  \subfloat[]{%
    \label{fig:advDiff3D_iters_N}%
    \includegraphics[width=0.48\textwidth]{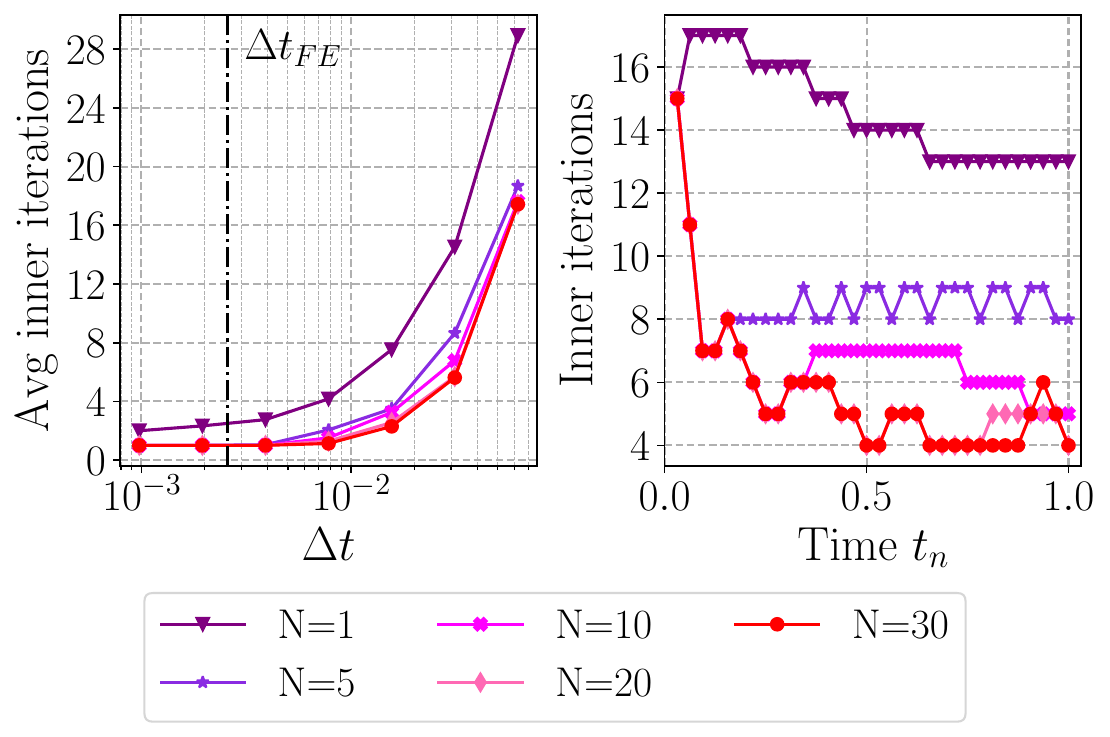}%
  }
  \caption{Number of inner iterations performed by the \imex method on the 3D advection–diffusion problem. In particular, we investigate how the average number of inner iterations varies for different values of \(\Delta t\), and the trend of iterations over time for \(\Delta t = 1/32\). \protect\subref{fig:advDiff3D_iters_eps} Results varying the values of the stability tolerance \(\varepsilon\). \protect\subref{fig:advDiff3D_iters_N} Results varying the reduced basis initial size \(N\). All tests were conducted setting $N_1=N_2=N_3=81$ and a maximal number of inner iterations $M=100$. When varying \(\varepsilon\), we fix \(N=10\); when varying \(N\), we fix \(\varepsilon = \bar{\varepsilon}\).}
  \label{fig:advDiff3D_subiterates_vs_epsilon_and_N}
\end{figure}

The average number of inner iterations also grows when employing smaller values of the initial basis size \(N\) (\cref{fig:advDiff3D_iters_N}). 
As expected, this behavior is mitigated when employing smaller values of \(\Delta t\).
Moreover, for values of \(N \geq 5\), the number of inner iterations remains comparable, both when averaging over varying \(\Delta t\) values and when fixing \(\Delta t = 1/32\) and examining their evolution over time. In the latter case, while for the initial timesteps some \(15\) to \(10\) iterations are performed, for the following timesteps \imex always performs around 8 iterations (\(N = 5\)) and 4 iterations (\(N = 30\)).  

The CPU times reported in \cref{fig:advDiff3D_N} show that the difference in terms of inner iterations does not justify employing \(N\) greater than 5: no further improvement in computational efficiency is noticed. Additionally, the convergence study of \cref{fig:advDiff3D_N} confirms that the achieved accuracy is, by construction of the method, independent of \(N\), as long as the absolute stability criterion is met. On the whole, for all values of \(\Delta t\) considered, the error made by the \imex scheme overlaps that of BE, and order 1 convergence is observed. 

\section{Conclusions}
\label{sec: conclusion}

In this paper, we present the first--order accurate \imex method, an implicit--explicit time integration scheme that exploits reduced bases. The method is shown to be convergent and unconditionally stable through both analytical derivations and numerical validation, yielding high--fidelity results even for time step sizes that exceed the forward Euler stability threshold. The method’s main hyperparameters --- namely, the stability tolerance $\varepsilon$ and the reduced basis dimension $N$ --- play a critical role in its performance.

Absolute stability is guaranteed if the tolerance \(\varepsilon\) is chosen according to \cref{theo:abs_stability}. However, since the bound in~\cref{eq:bound_eps} is often overly conservative, it can be relaxed in practice. For instance, in our numerical experiments on linear problems, we select $\varepsilon$ as inversely proportional to the system matrix’s condition number.

The reduced basis dimension $N$ must be tuned by the user based on the specific problem, striking a balance between stability and computational requirements. In our numerical tests, the reduced basis size is always chosen significantly smaller than the number of DOFs, which ensures low computational costs and modest storage requirements. 

Based on the obtained results, it is possible to identify a range of time step sizes --- exceeding the stability limit of forward Euler --- within which \imex yields computational advantages over backward Euler. The extent of this range is problem--dependent and influenced by both the the reduced basis size and the number of full--order DOFs. In particular, the 3D test case in \cref{subsec: advection diffusion 3D} shows a correlation between the efficiency gains of \imex and the number of DOFs, hinting at the method's scalability in high--dimensional settings.

Ongoing work aims at extending the proposed \imex framework to develop higher--order variants, leveraging the well--established class of IMEX schemes based on Runge--Kutta methods. 
Another important research direction involves incorporating mass matrices into the underlying ODE systems. This would seamlessly enable the application of \imex to PDEs discretized in space using Finite Element Methods (FEM) or, in general, approaches where a mass matrix naturally arises in the semi--discrete formulation. 
In this context, one relevant focus area is represented by domain decomposition. Indeed, we anticipate that \emph{ad hoc} treatments of the mass matrix could lead to drastically improved performance and better scalability properties in parallel computing environments.
Finally, we plan to assess \imex performances on more challenging and computationally intensive problems, such as the Navier--Stokes equations, the two--fluid turbulence plasma model \cite{braginskii1965Transport}, and gyrokinetic simulations~\cite{garbet2010gyrokinetic}. The associated high--dimensional systems, characterized by stiff dynamics and complex interactions, present ideal scenarios where the efficiency and adaptivity of \imex could be particularly advantageous.


\section*{Code availability}
The code necessary to reproduce the numerical results is available at \url{https://github.com/Fra-Sala/IMEX-RB}.

\section*{Acknowledgments}
SD and RT would like to thank the \textit{Swiss National Science Foundation} (grant agreement No 200021\_197021, ``Data--driven approximation
of hemodynamics by combined reduced order modeling and deep neural networks'').

\noindent Moreover, this work has been carried out within the framework of the EUROfusion Consortium, funded by the European Union via the Euratom Research and Training Programme (Grant Agreement No 101052200 — EUROfusion). Views and opinions expressed are however those of the author(s) only and do not necessarily reflect those of the European Union or the European Commission. Neither the European Union nor the European Commission can be held responsible for them.

\section*{Use of artificial intelligence}
The authors acknowledge the use of LLMs for grammar checking and language polishing (ChatGPT 4.0 and Grammarly\textsuperscript{\textregistered}) and take full responsibility for the content and any changes resulting from their use.

\printbibliography

\end{document}